\documentclass{amsart}

\usepackage{amssymb}

\usepackage{amsmath, amssymb, amsfonts, listings, hyperref, multicol, bm, shuffle,xcolor,enumerate,mathscinet,dsfont,tikz}

\usepackage[margin=1.3in]{geometry}

\usepackage{fancyhdr}
\newtheorem{theorem}{Theorem}[section]
\newtheorem{lemma}[theorem]{Lemma}

\theoremstyle{definition}
\newtheorem{definition}[theorem]{Definition}
\newtheorem{example}[theorem]{Example}

\theoremstyle{remark}
\newtheorem{remark}[theorem]{Remark}

\numberwithin{equation}{section}

\newtheorem{conjecture}[theorem]{Conjecture}
\newtheorem{proposition}[theorem]{Proposition}

\newcounter{casenum}

\usetikzlibrary{decorations.pathreplacing,shapes}
\usepackage{tikz}

\DeclareMathOperator{\Irr}{Irr}
\DeclareMathOperator{\Res}{Res}
\DeclareMathOperator{\ch}{\text{supp}} 
\DeclareMathOperator{\Char}{Char}
\DeclareMathOperator{\seq}{Seq}                                                                
\newcommand{\phcyc}[2]{\ensuremath{\p^{\Lambda_{#1}}_{#2}}}              
\newcommand{\ep}[1]{\varepsilon_{#1}}                                                       
\newcommand{\epch}[1]{\varepsilon^\vee_{#1}}                                           
\newcommand{\w}{\widetilde}                                                                       
\newcommand{\Li}{L}                                                                                    
\newcommand{\Ti}{T}                                                                                    
\newcommand{\T}[1]{\Ti \ensuremath{(#1)}}                                                 
\newcommand{\Tii}[2]{\ensuremath{\Ti_{#1;#2}}}                                          
\newcommand{\Sign}{S}                                                                               
\newcommand{\Si}[1]{\Sign \ensuremath{(#1)}}                                            
\newcommand{\Sii}[2]{\ensuremath{\Sign_{#1;#2}}}                                      
\DeclareMathOperator{\Repr}{Rep}                                                              
\newcommand{\rep}[1]{\Repr \ensuremath{^{#1}}}                                        
\newcommand{\Rcal}[1]{\mathcal{R}(#1)}                                                     
\newcommand{\Rcalj}[2]{\mathcal{R}_{#1}(#2)}                                            
\DeclareMathOperator{\pro}{pr}                                                                    
\newcommand{\pr}[1]{\pro \ensuremath{_{\Lambda_{#1}}}}                          
\DeclareMathOperator{\infltxt}{infl}                                                                
\newcommand{\infl}[1]{\infltxt \ensuremath{_{\Lambda_{#1}}}}                     
\newcommand{\e}[1]{\ensuremath{e_{#1}}}                                                   
\newcommand{\etil}[1]{\ensuremath{\w{e}_{#1}}}                                          
\newcommand{\ech}[1]{\ensuremath{e^\vee_{#1}}}                                      
\newcommand{\etilch}[1]{\ensuremath{\w{e}^\vee_{#1}}}                              
\newcommand{\ftil}[1]{\ensuremath{\w{f}_{#1}}}                                            
\newcommand{\ftilch}[1]{\ensuremath{\w{f}^\vee_{#1}}}                               
\newcommand{\und}[1]{\underline{\bf{#1}}}                                                   
\newcommand{\jump}[1]{\text{jump}_{#1}}                                                    
\newcommand{\wti}[1]{\text{wt}_{#1}}                                                           
\newcommand{\Sy}[1]{\ensuremath{\mathcal{S}_{#1}}}                               
\newcommand{\cycloI}[1]{\ensuremath{\mathcal{I}^{#1}}}                            
\newcommand{\maps}{\colon}                                                                       
\def\Shuffle{\,\raise 1pt\hbox{$\scriptscriptstyle\cup{\mskip                           
               -4mu}\cup$}\,}
\newcommand{\supp}[1]{\text{supp}(\ensuremath{#1})}                                 
\newcommand{\gammaplus}[2]{\ensuremath{\gamma^+_{#1;#2}}}                
\newcommand{\gammaminus}[2]{\ensuremath{\gamma^-_{#1;#2}}}              
\newcommand{\Lii}[1]{\ensuremath{\Li(#1)}}                                                   
\newcommand{\crystalmap}{\mathcal{T}}                                                       
\newcommand{\crystalmapopp}{\crystalmap^{\mathrm{opp}}} 
\newcommand{\cryphi}[1]{\p_{#1}}                                                                  
\newcommand{\Ccal}[1]{\mathcal{C}(#1)}                                                       

\newcommand{\Bp}{\mathcal{B}}  
\newcommand{\Bkr}{{B}^{1,1}}   
\newcommand{\Bkropp}{{B}^{\ell-1,1}}   
\newcommand{\Bopp}{\Bp^{\mathrm{opp}}} 	 
\newcommand{\F}{\mathbb{F}}                            	 
\newcommand{\Z}{\mathbb{Z}}                           	 
\newcommand{\N}{\mathbb{N}}                           	 
\newcommand{\Q}{\mathbb{Q}}                           	 
\newcommand{\wh}{\widehat}                              	 
\newcommand{\Fr}{\mathfrak}                            		 
\newcommand{\MB}{\mathbb}                            		 
\newcommand{\p}{\varphi}                                 		 
\DeclareMathOperator{\Hom}{Hom}                    	 
\DeclareMathOperator{\ind}{Ind}                          	 
\DeclareMathOperator{\res}{Res}                        	 
\DeclareMathOperator{\soc}{soc}                        	 
\DeclareMathOperator{\cosoc}{cosoc}                	 
\DeclareMathOperator{\Mod}{\dash \text{mod}}          
\DeclareMathOperator{\wt}{wt}                           		 
\DeclareMathOperator{\HOM}{HOM}                 	 	 
\DeclareMathOperator{\dash}{-}                        		 
\DeclareMathOperator{\UnitModule}{\mathds{1}}         
\newcommand{\zero}{\bm{0}}                                       

\newcommand{\omitt}[1]{}

\usepackage{etex} 

\tikzstyle{new}=[fill =white, circle, inner sep=.5pt]
\tikzstyle{old}=[fill =blue!20, circle, inner sep=.5pt]

\newcommand{\Part}[1]{
 \foreach \xx [count=\ss from 1] in {#1}{
 	{\ifnum\ss=1
		\draw (0,\ss-1)--(\xx,\ss-1); 
		\fi}
   \draw (0,\ss) to (\xx,\ss);
   \foreach \y in {0, ..., \xx} {\draw (\y,\ss)--(\y,\ss-1);}
 }}

 \newcommand{\fPart}[2]{
 \foreach \xx [count=\ss from 1] in {#1}{
  \filldraw[#2] (0,\ss) -- (\xx,\ss)--(\xx,\ss-1)--(0,\ss-1);
 }
 \foreach \xx [count=\ss from 1] in {#1}{
 	{\ifnum\ss=1
		\draw (0,\ss-1)--(\xx,\ss-1); 
		\fi}
   \draw (0,\ss) to (\xx,\ss);
   \foreach \y in {0, ..., \xx} {\draw (\y,\ss)--(\y,\ss-1);}
 }}

\def\UNIT{.2} 

\newcounter{r}
\newcounter{c}

\newcommand\youngDiagram[2]{
        \setcounter{r}{0}
        \foreach \row in {#1} {
                \setcounter{c}{0}
                \foreach \b in \row {
                        \node at (\value{c}*#2 + .5*#2, \value{r}*#2+.5*#2)
{\b};
                        \addtocounter{c}{1}

                }
                \draw[step = {#2}] (0,\value{r}*#2) grid
(\value{c}*#2,\value{r}*#2+#2);
                \addtocounter{r}{-1}
        }
}

\def\Lattice{
	\coordinate (0) at (0,0);
	\coordinate (11) at (0,1.8);
	\coordinate (21) at (-1.2,3.8);\coordinate (22) at (1.2,3.9);
	\coordinate (31) at (-1.2,6.1);\coordinate (32) at (1.2,6.1);
	\foreach \xx in {1, ..., 4}{\coordinate (4\xx) at (-7.1+2.9*\xx,8.5);}

	\draw (0)--(11) (11)--(21) (11)--(22);
	\draw (21)--(31) (22)--(32);
	\draw (31)--(41)  (31)--(42)  (32)--(43)  (32)--(44);
\begin{scope}[every node/.style={fill=white}]
	\node at (0) {$\emptyset$};
	\node at (11) {};
	\node at (21) {};
	\node at (31) {};
	\node at (41) {};
	\node at (22) {};
	\node at (32) {};
	\node at (42) {};
	\node at (43) {};
	\node at (44) {};
\end{scope}
\foreach \xx in {-3,0,3} {\node at (\xx , 9.5) {$\vdots$}; }
}

\def\LatticeMV{
        \coordinate (0) at (0,0);
        \coordinate (11) at (0,1.8);
        \coordinate (21) at (-1.2,3.8);\coordinate (22) at (1.2,3.9);
        \coordinate (31) at (-1.2,6.1);\coordinate (32) at (1.2,6.1);
        \foreach \xx in {1, ..., 4}{\coordinate (4\xx) at (-7.1+2.9*\xx,8.5);}

        \draw (0)--(11) (11)--(21) (11)--(22);
        \draw (21)--(31) (22)--(32);
\begin{scope}[every node/.style={fill=white}]
        \node at (0) {$\emptyset$};
        \node at (11) {};
        \node at (21) {};
        \node at (31) {};
        \node at (22) {};
        \node at (32) {};
\end{scope}
\foreach \xx in {-3,0,3} {\node at (\xx , 7.5) {$\vdots$}; }
}

\def\LatticeHK{
        \coordinate (0) at (0,0);
        \coordinate (11) at (0,1.8);
        \coordinate (21) at (-1.8,3.8);\coordinate (22) at (1.8,3.8);
        \coordinate (31) at (-1.8,6.1);\coordinate (32) at (1.8,6.1);
        \foreach \xx in {1, ..., 3}{\coordinate (4\xx) at (-8.5+3.4*\xx,8.9);}
        \coordinate (41) at (-5,8.9);
        \coordinate (42) at (-1.3,8.9);
        \coordinate (43) at (1.6,8.9);
        \coordinate (44) at (4,8.9);

        \draw (0)--(11) (11)--(21) (11)--(22);
        \draw (21)--(31) (22)--(32);
        \draw (31)--(41)  (31)--(42)  (32)--(43)  (32)--(44);
\begin{scope}[every node/.style={fill=white}]
        \node at (0) {$\emptyset$};
        \node at (11) {};
        \node at (21) {};
        \node at (31) {};
        \node at (41) {};
        \node at (22) {};
        \node at (32) {};
        \node at (42) {};
        \node at (43) {};
        \node at (44) {};
\end{scope}
\foreach \xx in {-3,0,3} {\node at (\xx , 10) {$\vdots$}; }
}

\def\LatticeHKWider{
        \coordinate (0) at (0,0);
        \coordinate (11) at (0,1.8);
        \coordinate (21) at (-1.8,3.8);\coordinate (22) at (2,3.8);
        \coordinate (31) at (-1.8,6.1);\coordinate (32) at (2,6.1);
        \coordinate (41) at (-5.2,8.9); \coordinate (42) at (-.7,8.9); \coordinate (43) at (3,8.9); \coordinate (44) at (6.5,8.9);

        \draw (0)--(11) (11)--(21) (11)--(22);
        \draw (21)--(31) (22)--(32);
        \draw (31)--(41)  (31)--(42)  (32)--(43)  (32)--(44);
\begin{scope}[every node/.style={fill=white}]
        \node at (0) {$\emptyset$};
        \node at (11) {};
        \node at (21) {};
        \node at (31) {};
        \node at (41) {};
        \node at (22) {};
        \node at (32) {};
        \node at (42) {};
        \node at (43) {};
        \node at (44) {};
\end{scope}
\foreach \xx in {-3,0,3} {\node at (\xx , 10) {$\vdots$}; }
}

\begin{document}

\title[Categorifying 
$\Bp \otimes B(\Lambda_i)$
]{Categorifying the tensor product of a level 1 highest weight and perfect crystal in type $A$}


\author{Monica Vazirani}\thanks{This work was partially supported
by NSF-MSPRF, NSA grant H98230-12-1-0232, and the Simons Foundation.}
\address{Mathematics Department, One Shields Ave, Davis, CA 95616}
\email{}
\thanks{}


\omitt{
The Mathematics Subject Classification numbers are primarily
81R50; 05E10 and secondarily  17B37;  20C08.
key words include ``Khovanov-Lauda-Rouquier algebras" and "quiver
Hecke algebras."
crystal graphs, categorification
}

\date{}

\begin{abstract}
We use KLR algebras to categorify a crystal isomorphism
between a highest weight crystal and the tensor product of a perfect
crystal and another highest weight crystal, all in level $1$ type $A$
affine. 
 The nodes 
of the perfect crystal correspond
to a family of trivial modules 
and the nodes of 
the highest weight crystal 
 correspond to simple modules, which we may also parameterize
by $\ell$-restricted partitions.
In the case $\ell$ is a prime, one can reinterpret all the
results for the symmetric group in characteristic $\ell$.
The crystal operators correspond to socle of restriction
and behave compatibly with the rule for tensor product
of crystal graphs.
\end{abstract}

\maketitle

\tableofcontents

\section{Introduction}
\label{sec:in}
Kang-Kashiwara \cite{KK12} and Webster \cite{Web}
show the cyclotomic Khovanov-Lauda-Rouquier (KLR) algebra $R^\Lambda$
categorifies the highest weight representation $V(\Lambda)$
in arbitrary symmetrizable type.
(KLR algebras are also known as quiver Hecke algebras.)
We will say the {\em combinatorial} version of this statement is that
$R^\Lambda$ categorifies the crystal $B(\Lambda)$, where
simple modules correspond to nodes, and functors that take socle
of restriction correspond to arrows, i.e.~the Kashiwara crystal
operators. 
Webster \cite{Web} and Losev-Webster \cite{LW} 
categorify the tensor product of highest weight modules,
and hence the tensor product of highest weight crystals.
However, one can consider a tensor product of crystals
\begin{gather}
\label{eq-tensor}
\Bp \otimes B(\Lambda) \simeq B(\Lambda')
\end{gather}
where $\Lambda, \Lambda' \in P^+$ are of level $k$ and
$\Bp$ is a perfect crystal of level $k$.
In this paper, we (combinatorially) categorify the crystal isomorphism 
\eqref{eq-tensor} in the case
the level $k=1$ for type $A_{{\ell-1}}^{(1)}$ and $\Bp = \Bkr$
which is drawn in Figure \ref{fig-perfect-crystal}.
Each node of $\Bp$ corresponds to a family of trivial modules,
but note this does {\em not} give a categorification of $\Bp$.
(By symmetry we have similar results for $\Bp = B^{\ell-1,1}$
of Figure \ref{fig-Bopp}
whose nodes correspond to sign modules.)

We note that this gives a construction of simple modules that
is somewhat intermediate between the crystal operator construction
and the Specht module construction.  Combinatorially, the former
corresponds to building an $\ell$-restricted
partition one (good) box at a time.
Our construction builds a partition one row at a time,
or dually one column at a time.
The Specht module construction (at least for $\F_\ell \Sy{n}$
or the Hecke algebra of type $A$) builds the simple from the whole
partition,  constructing the simple as a
subquotient of an induced trivial module from a parabolic
subalgebra that corresponds to the partition.
However, this paper also describes how socle of restriction
interacts with the construction.
One can also recover this construction for finite type
$A_{\ell-1}$ as its Dynkin diagram is a subdiagram of that
of type $A^{(1)}_{\ell-1}$, or recovers characteristic $0$
constructions taking $\ell \to \infty$.
For a construction of simple modules
related to the crystal $B(\infty)$
for finite type KLR algebras see \cite{BKOP}.

This paper is based on unpublished 
work of the author \cite{V03,Vnotes}
which was done for the affine Hecke algebra of type $A$
at an $\ell$th root of unity.
We chose to rewrite this in the language of KLR algebras
to appeal to the modern reader and also make it easier
to then generalize the theorem to other affine types in
\cite{KV}.

I wish to thank Henry Kvinge for his help with the figures
and whose feedback greatly improved the exposition.

\section{Type $A$ Cartan datum and crystals}

\subsection{Cartan datum for type $A^{(1)}_{\ell-1}$}
Fix an integer $\ell \geq 2$. In this paper we will work solely in type $A^{(1)}_{\ell-1}$. 
\begin{figure}
\begin{center}
\begin{tikzpicture}
\draw (0,0) circle (.1cm);
\draw (2,0) circle (.1cm);
\node at (0,-.75) {$0$};
\node at (2,-.75) {$1$};
\draw (0.1,.05) -- (1.9,.05);
\draw (0.1,-.05) -- (1.9,-.05);
\draw (0.2,0) -- (0.3,.2);
\draw (0.2,0) -- (0.3,-.2);
\draw (1.8,0) -- (1.7,.2);
\draw (1.8,0) -- (1.7,-.2);
\end{tikzpicture}
\qquad 
\qquad 
\begin{tikzpicture} [scale=.7]
\draw (0,0) circle (.1cm);
\node at (0,-.75) {$1$};
\draw (1.5,0) circle (.1cm);
\node at (1.5,-.75) {$2$};
\draw (3,0) circle (.1cm);
\node at (3,-.75) {$3$};
\draw (6,0) circle (.1cm);
\node at (6,-.75) {$\ell-2$};
\draw (7.5,0) circle (.1cm);
\node at (7.5,-.75) {$\ell-1$};
\draw (3,2) circle (.1cm);
\node at (3.5,2.5) {$0$};
\draw (.071,.071) -- (2.9,1.95);
\draw (.1,0) -- (1.4,0);
\draw (1.6,0) -- (2.9,0);
\draw (3.1,0) -- (3.7,0);
\draw (5.3,0) -- (5.9,0);
\draw (6.1,0) -- (7.4,0);
\draw (7.429,0.071) -- (3.1,1.95);
\node at (4.1,0) {$\cdot$};
\node at (4.5,0) {$\cdot$};
\node at (4.9,0) {$\cdot$};
\end{tikzpicture}
\end{center}
\caption{
\label{figure-dynkin}
The Dynkin diagram for
$A^{(1)}_{1}$ is on the left and
and the Dynkin diagram for $A^{(1)}_{\ell-1}$ with $\ell > 2$ is
on the right.}
\end{figure}
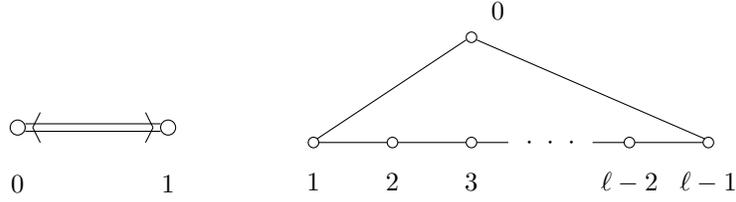
Let $I$ be the indexing set
\begin{equation}
I = \{ 0, 1, \dots, \ell-1\}.
\end{equation}
Let $[a_{ij}]_{i,j \in I}$ denote the associated Cartan matrix. 
For $\ell > 2$ the type $A$ Cartan matrix is the $\ell \times \ell$
matrix
\begin{equation*}
\left[
\begin{array}{rrrrrrr}
2& -1& 0 & &\cdots  &0&0 \\
-1& 2 & -1 && \cdots &0& 0 \\
\vdots& & &\ddots & &&\vdots \\
0& 0& & \cdots& -1&2& -1 \\
0& 0& & \cdots& 0&-1& 2
\end{array}
\right].
\end{equation*}
When $\ell = 2$ it is 
\begin{equation*}
\left[
\begin{array}{rr}
2&-2 \\
-2&2
\end{array}
\right].
\end{equation*}

Following \cite{Kac85} we let $\Fr{h}$ be a Cartan subalgebra, $\prod = \{\alpha_0, \dots, \alpha_{\ell-1}\}$ its system of simple roots, $\prod^{\vee} = \{h_0, \dots, h_{\ell-1}\}$ its simple coroots, and $Q$ and $Q^\vee$ the root and coroot lattices respectively. Then set
\begin{equation}
Q^+ = \bigoplus_{i \in I} \mathbb{Z}_{\geq 0} \alpha_i.
\end{equation}
For an element $\nu \in Q^+$, we define its {\emph{height}}, $|\nu|$, to be the sum of the coefficients, i.e. if $\nu = \sum_{i \in I} \nu_i \alpha_i$ then
\begin{equation}
|\nu| = \sum_{i \in I} \nu_i.
\end{equation}
We also have a symmetric bilinear form
\begin{equation*}
(\; , \; ) : \Fr{h}^* \times \Fr{h}^* \rightarrow \mathbb{C}
\end{equation*}
which satisfies
\begin{equation}
a_{ij} = \langle h_i, \alpha_j \rangle =  
\frac{2 (\alpha_i, \alpha_j)}{ (\alpha_i, \alpha_i)}
\end{equation}
where $\langle \; , \; \rangle: \Fr{h} \times \Fr{h}^* \rightarrow \mathbb{C}$ is the canonical pairing. Using this pairing we define the fundamental weights $\{\Lambda_i \; | \; i \in I\}$ via
\begin{equation*}
\langle h_j, \Lambda_i \rangle = \delta_{ji}.
\end{equation*}
The weight lattice is $ \bigoplus_{i \in I} \MB{Z} \Lambda_i $ and the integral dominant weights are
\begin{equation*}
P^+ = \bigoplus_{i \in I} \mathbb{Z}_{\geq 0} \Lambda_i.
\end{equation*}

\begin{remark}
\label{rem-k}
Because in this paper we work exclusively with Cartan datum associated with $A^{(1)}_{\ell-1}$, it is often convenient to identify elements of $I$ with $\mathbb{Z}/\ell \mathbb{Z}$, so when stating that $k \in I$, we will usually think of $k \in  \mathbb{Z}/\ell \mathbb{Z}$ even if we neglect to write $\overline{k}$ or $k \bmod \ell$. 
We will often be considering a sequence of $k$ operators,
$k \in \N := \Z_{\ge 0}$,
but the $k$th operator may be indexed by $(k-1)\bmod \ell$, for which
it is convenient to relax notation.
\end{remark}

%
\subsection{Review of crystals}
\label{sec-crystal-review}
%

We recall the tensor category of crystals following
Kashiwara \cite{Kas95}, see also \cite{Kas90b,Kas91,KS97}.

A {\emph{crystal}} is a set $B$ together with maps
\begin{itemize}
  \item $\wt \maps B \longrightarrow P$,
  \item $\ep{i}, \p_i : B \longrightarrow \MB{Z} \sqcup \{-\infty\}$ \quad for $i \in I$,
  \item $\etil{i}, \ftil{i} \maps B \longrightarrow B \sqcup \{0\}$ \quad for $i \in I$,
\end{itemize}
such that
\begin{enumerate}[C1.]
 \item $\p_{i}(b) =\ep{i}(b)+ \langle h_i, \wt(b) \rangle$ \quad for any $i$.
 \item If $b \in B$ satisfies $\etil{i} b \neq 0$, then
  \begin{align*}
 &  \ep{i}(\etil{i}b) = \ep{i}(b)-1, & \p_{i}(\etil{i}b) = \p_{i}(b) +1, & & \wt(\etil{i}b) = \wt(b)+\alpha_i.
  \end{align*}
 \item If $b \in B$ satisfies $\ftil{i} b \neq 0$, then
  \begin{align*}
 &  \ep{i}(\ftil{i}b) = \ep{i}(b)+1,
 & \p_{i}(\ftil{i}b) = \p_{i}(b)-1,
 & &\wt(\ftil{i}b) = \wt(b)-\alpha_i.
  \end{align*}
 \item For $b_1$, $b_2 \in B$, $b_2=\ftil{i}b_1$ if and only if
$\etil{i}b_2 = b_1$.
 \item If $\p_{i}(b) = -\infty$, then $\etil{i}b=\ftil{i}b=0$.
\end{enumerate} \medskip

If $B_1$ and $B_2$ are two crystals, then a {\emph{morphism}} $\psi \maps B_1 \rightarrow B_2$ of crystals is a map $$\psi \maps B_1 \sqcup \{0\} \rightarrow B_2 \sqcup \{0\}$$ satisfying the following properties:
\begin{enumerate}[M1.]
  \item $\psi(0) = 0$.
  \item If $\psi(b) \neq 0$ for $b \in B_1$, then
\begin{align*}
  & \wt(\psi(b)) = \wt(b),
  & \ep{i}(\psi(b)) = \ep{i}(b),
  & &\p_{i}(\psi(b)) = \p_{i}(b).
\end{align*}
 \item For $b \in B_1$ such that $\psi(b) \neq 0$ and $\psi(\etil{i}b) \neq 0$, we have $\psi(\etil{i}b) = \etil{i}(\psi(b))$.
 \item For $b \in B_1$ such that $\psi(b) \neq 0$ and $\psi(\ftil{i}b) \neq 0$, we have $\psi(\ftil{i}b) = \ftil{i}(\psi(b))$.
\end{enumerate}
A morphism $\psi$ of crystals is called {\emph{strict}} if
\begin{equation*}
  \psi \circ \etil{i} = \etil{i} \circ \psi, \qquad \quad
\psi  \circ \ftil{i} = \ftil{i} \circ \psi,
\end{equation*}
and an {\emph{embedding}} if $\psi$ is injective. \medskip

Given two crystals $B_1$ and $B_2$ their tensor product
$B_1 \otimes B_2$ (using the reverse Kashiwara convention)
 has underlying set $\{b_1 \otimes b_2; b_1 \in B_1, \;
\text{and} \; b_2 \in B_2 \}$ where we identify $b_1 \otimes 0 = 0
\otimes b_2 = 0$. The crystal structure is given as follows:
\begin{align} \label{tensor-product-crystals}
  & \wt(b_1 \otimes b_2) = \wt(b_1) + \wt (b_2), \\
  & \ep{i}(b_1 \otimes b_2) = \max\{ \ep{i}(b_2), \ep{i}(b_1)-\langle h_i, \wt(b_2) \rangle\},
\\ &
\cryphi{i}(b_1 \otimes b_2) = \max\{ \cryphi{i}(b_2) + \langle h_i,\wt(b_1) \rangle, \cryphi{i}(b_1)\}, 
\end{align}
\begin{align}
  &\etil{i}(b_1 \otimes b_2 ) = 
  \begin{cases}
    \etil{i}b_1 \otimes b_2 & \text{if $\ep{i}(b_1) > \cryphi{i}(b_2)$}\\
    b_1 \otimes \etil{i}b_2 & \text{if $\ep{i}(b_1) \leq \cryphi{i}(b_2)$},
  \end{cases}
\label{eq_ei_tensor}
\\
& \ftil{i}(b_1 \otimes b_2 ) =
  \begin{cases}
    \ftil{i}b_1 \otimes b_2 &  \text{if $\ep{i}(b_1) \geq \cryphi{i}(b_2)$}\\
    b_1 \otimes \ftil{i}b_2 & \text{if $\ep{i}(b_1) < \cryphi{i}(b_2)$}.
\label{eq_fi_tensor}
  \end{cases}
\end{align}

Given a crystal $B$, we can draw its  associated crystal graph
with nodes (or vertices) $B$ and $I$-colored arrows (directed edges)
as follows.
When $\etil{i}b = a$ (so $b = \ftil{i}a$) we draw an $i$-colored arrow $a \xrightarrow{i} b$. We also say $b$ has an incoming $i$-arrow and $a$ has an outgoing $i$-arrow.

\section{Level $1$ crystals in type $A^{(1)}_{\ell-1}$}
\label{sec-crystal1}

\begin{figure} 
\begin{center}
\begin{tikzpicture} [scale=.5]
\draw (0,.25) ellipse (.65cm and .32cm);
\draw (3,.25) ellipse (.65cm and .32cm);
\draw (8,.25) ellipse (.65cm and .32cm);
\draw (11,.25) ellipse (.65cm and .32cm);
\draw[->,thick] (.8,.25) -- (2.2,.25);
\draw[->,thick] (8.8,.25) -- (10.1,.25);
\draw[->,thick] (3.8,.25) -- (4.6,.25);
\draw[->,thick] (6.3,.25) -- (7.2,.25);
\draw [fill] (5,0.25) circle [radius=0.03];
\draw [fill] (5.5,0.25) circle [radius=0.03];
\draw [fill] (6,0.25) circle [radius=0.03];
\node at (0,.25) {\scriptsize{0}};
\node at (3,.25) {\scriptsize{1}};
\node at (8,.25) {\scriptsize{$\ell$-2}};
\node at (11,.25) {\scriptsize{$\ell$-1}};
\node at (1.4,.7) {\scriptsize{1}};
\node at (4.2,.7) {\scriptsize{2}};
\node at (6.7,.7) {\scriptsize{$\ell$-2}};
\node at (9.4,.7) {\scriptsize{$\ell$-1}};
\node at (5.5,3.5) {\scriptsize{0}};
\draw[->,thick] (10.8,.75) to [out = 130, in = 50] (0.1,.75);
\end{tikzpicture}
\end{center}
\caption{\label{fig-perfect-crystal}
The level $1$ perfect  crystal $\Bp$, which is also denoted $B^{1,1}$.}
\end{figure}

The level $1$ highest weight crystal, or fundamental crystal,
 $B(\Lambda_i)$ has a model
(see Figure \ref{fig-fundamental}) with nodes $\ell$-{\em restricted}
partitions, i.e. $\lambda = (\lambda_1, \ldots, \lambda_t)$
such that $\lambda_r \in \Z_{\ge 0}$, 
$0 \le \lambda_{r} - \lambda_{r+1} < \ell$ for all $r$.
Observe that for fixed $\ell$, as directed graphs
$B(\Lambda_0)$ and $B(\Lambda_i)$ are identical. 
The edge labels or ``colors" for $B(\Lambda_i)$ are
obtained from those of $B(\Lambda_0)$ by adding $i \bmod \ell$.

Let $\Bp$ be the crystal graph in Figure \ref{fig-perfect-crystal}.
$\Bp$ is an example of a level $1$ 
\emph{perfect} crystal. See \cite{KKMMNN} for the definition
of a perfect crystal and for many of its important properties.
$\Bp$, often denoted $B^{1,1}$ in the literature
 is also an example of a Kirillov-Reshetikhin crystal.
Observe that we have parameterized the nodes of $\Bp$ so that 
\begin{equation*}
\ep{i}(
\begin{tikzpicture}[baseline=-2pt]
        \node at (0,0) {\;\;$k$\;\;};
        \draw (0,0) ellipse (.38cm and .2cm);
\end{tikzpicture}
) = \delta_{i,k}.
\end{equation*}

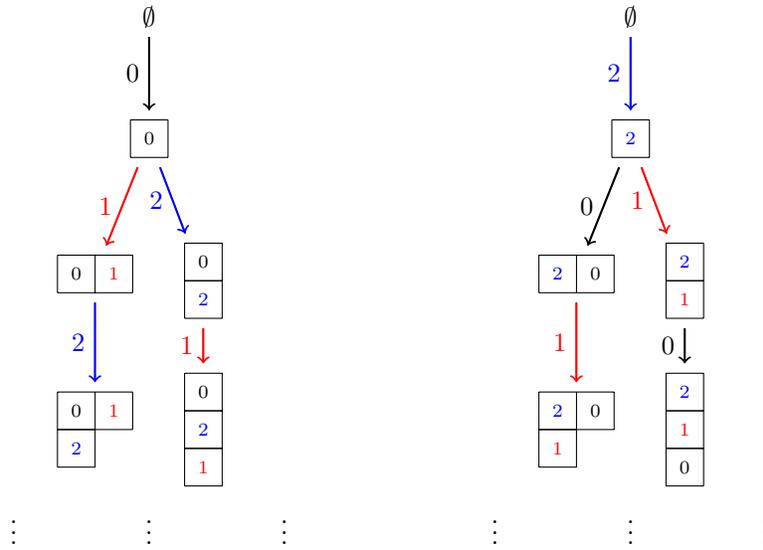
\begin{figure}[ht]
\begin{center}

\begin{tikzpicture} [scale=.08]
\node at (0,0) {\begin{tikzpicture}[xscale=3*\UNIT, yscale=-4.5*\UNIT]
\begin{scope}
        \LatticeMV 
\end{scope} 
\begin{scope}[every node/.style={fill=white}]
        \node (a) at (0) {$\emptyset$};
        \node (b) at (11) {\begin{tikzpicture}\youngDiagram{{\scriptsize{0}}}{.5};\end{tikzpicture}};
        \node (c) at (21) {\begin{tikzpicture}\youngDiagram{{\scriptsize{0},\scriptsize{\color{red}{1}}}}{.5};\end{tikzpicture}};
        \node (d) at (31) {\begin{tikzpicture}\youngDiagram{{\scriptsize{0},\scriptsize{\color{red}{1}}},{\scriptsize{\color{blue}{2}}}}{.5};\end{tikzpicture}};
        \node (g) at (22) {\begin{tikzpicture}\youngDiagram{{\scriptsize{0}},{\scriptsize{\color{blue}{2}}}}{.5};\end{tikzpicture}};
        \node (i) at (32) {\begin{tikzpicture}\youngDiagram{{\scriptsize{0}},{\scriptsize{\color{blue}{2}}},{\scriptsize{\color{red}{1}}}}{.5};\end{tikzpicture}};
\end{scope}
\begin{scope}[black, thick,->]
\draw (a)--(b) node[midway,left] {0};
\draw[red] (b)--(c) node[midway,left] {\color{red}{1}};
\draw[blue] (b)--(g) node[midway,left] {\color{blue}{2}};
\draw[red] (g)--(i) node[midway,left] {\color{red}{1}};
\draw[blue] (c)--(d) node[midway,left] {\color{blue}{2}};
\end{scope}
\end{tikzpicture}};

\node at (80,0) {\begin{tikzpicture}[xscale=3*\UNIT, yscale=-4.5*\UNIT]
\begin{scope}
        \LatticeMV 
\end{scope} 
\begin{scope}[every node/.style={fill=white}]
        \node (a) at (0) {$\emptyset$};
        \node (b) at (11) {\begin{tikzpicture}\youngDiagram{{\scriptsize{\color{blue}{2}}}}{.5};\end{tikzpicture}};
        \node (c) at (21) {\begin{tikzpicture}\youngDiagram{{\scriptsize{{\color{blue}{2}}},\scriptsize{\color{black}{0}}}}{.5};\end{tikzpicture}};
        \node (d) at (31) {\begin{tikzpicture}\youngDiagram{{\scriptsize{{\color{blue}{2}}},\scriptsize{\color{black}{0}}},{\scriptsize{\color{red}{1}}}}{.5};\end{tikzpicture}};
        \node (g) at (22) {\begin{tikzpicture}\youngDiagram{{\scriptsize{{\color{blue}{2}}}},{\scriptsize{\color{red}{1}}}}{.5};\end{tikzpicture}};
        \node (i) at (32) {\begin{tikzpicture}\youngDiagram{{\scriptsize{{\color{blue}{2}}}},{\scriptsize{\color{red}{1}}},{\scriptsize{\color{black}{0}}}}{.5};\end{tikzpicture}};
\end{scope}
\begin{scope}[black, thick,->]
\draw[blue] (a)--(b) node[midway,left] {{\color{blue}{2}}};
\draw (b)--(c) node[midway,left] {\color{black}{0}};
\draw[red] (b)--(g) node[midway,left] {\color{red}{1}};
\draw (g)--(i) node[midway,left] {\color{black}{0}};
\draw[red] (c)--(d) node[midway,left] {\color{red}{1}};
\end{scope}
\end{tikzpicture}};

\end{tikzpicture}

\caption{ \label{fig-fundamental}
$B(\Lambda_0)$ and $ B(\Lambda_2)$
for $\ell = 3$.
}
\end{center}
\end{figure}

 Then
$\crystalmap:B(\Lambda_{i})
\xrightarrow{\simeq}
 \Bp \otimes B(\Lambda_{i-1})$
is an isomorphism of crystals. The isomorphism is pictured in Figure
\ref{fig-crystal-iso} for $i = 0$ and $\ell = 3$.
Combinatorially, $\crystalmap(\lambda) = $
\begin{tikzpicture}[baseline=-2pt] 
	\node at (0,0) {\;\;$k$\;\;}; 
        \draw (0,0) ellipse (.38cm and .2cm);
\end{tikzpicture}$ \; \otimes\; \mu$
where
$k\equiv \lambda_1 +i-1 \bmod \ell$
and $\mu = (\lambda_2,
\dots, \lambda_t)$ if $\lambda =(\lambda_1, \lambda_2, \cdots,
\lambda_t)$.
So we obtain $\mu$ from $\lambda$ by removing its top row.
In Figure \ref{fig-crystal-iso}, we draw
\begin{equation}
\begin{tikzpicture}
			\node at (0,0) {$\crystalmap(\lambda) = $};
			\node (z) at (1,.3) {\scriptsize{$k$}};
                       	\draw (z) ellipse (.35cm and .15cm);
			\node at (1,0) {$\otimes$};
			\node at (1,-.3) {$\mu$};
\end{tikzpicture}
\end{equation}
so the visual of the top row removal stands out.
Note $\lambda_1 - \mu_1 = \lambda_1 - \lambda_2 < \ell$
means that $\crystalmap$ has a well-defined inverse.


\begin{figure}[ht]
\begin{center}
\begin{tikzpicture} [scale=.018]
\node at (0,0) {\begin{tikzpicture}[xscale=3*\UNIT, yscale=-4.5*\UNIT]
\begin{scope}
	\Lattice 
\end{scope} 
\begin{scope}[every node/.style={fill=white}]
	\node (a) at (0) {$\emptyset$};
	\node (b) at (11) {\begin{tikzpicture}\youngDiagram{{\scriptsize{0}}}{.35};\end{tikzpicture}};
	\node (c) at (21) {\begin{tikzpicture}\youngDiagram{{\scriptsize{0},\scriptsize{\color{red}{1}}}}{.35};\end{tikzpicture}};
	\node (d) at (31) {\begin{tikzpicture}\youngDiagram{{\scriptsize{0},\scriptsize{\color{red}{1}}},{\scriptsize{\color{blue}{2}}}}{.35};\end{tikzpicture}};
	\node (e) at (41) {\begin{tikzpicture}\youngDiagram{{\scriptsize{0},\scriptsize{\color{red}{1}}},{\scriptsize{\color{blue}{2}},\scriptsize{0}}}{.35};\end{tikzpicture}};
	\node (g) at (22) {\begin{tikzpicture}\youngDiagram{{\scriptsize{0}},{\scriptsize{\color{blue}{2}}}}{.35};\end{tikzpicture}};
	\node (i) at (32) {\begin{tikzpicture}\youngDiagram{{\scriptsize{0}},{\scriptsize{\color{blue}{2}}},{\scriptsize{\color{red}{1}}}}{.35};\end{tikzpicture}};
	\node (k) at (42) {\begin{tikzpicture}\youngDiagram{{\scriptsize{0},\scriptsize{\color{red}{1}},\scriptsize{\color{blue}{2}}},{\scriptsize{\color{blue}{2}}}}{.35};\end{tikzpicture}};
	\node (l) at (43) {\begin{tikzpicture}\youngDiagram{{\scriptsize{0},\scriptsize{\color{red}{1}}},{\scriptsize{\color{blue}{2}}},{\scriptsize{\color{red}{1}}}}{.35};\end{tikzpicture}};
	\node (m) at (44) {\begin{tikzpicture}\youngDiagram{{\scriptsize{0}},{\scriptsize{\color{blue}{2}}},{\scriptsize{\color{red}{1}}},{\scriptsize{0}}}{.35};\end{tikzpicture}};
\end{scope}
\begin{scope}[black, thick,->]
\draw (a)--(b) node[midway,left] {0};
\draw[red] (b)--(c) node[midway,left] {\color{red}{1}};
\draw (d)--(e) node[midway,left] {0};
\draw[blue] (b)--(g) node[midway,left] {\color{blue}{2}};
\draw[red] (g)--(i) node[midway,left] {\color{red}{1}};
\draw[blue] (c)--(d) node[midway,left] {\color{blue}{2}};
\draw[blue] (d)--(k) node[midway,left] {\color{blue}{2}};
 \draw[red] (i)--(l) node[midway,left] {\color{red}{1}};
\draw (i)--(m) node[midway,left] {0};
\end{scope}
\end{tikzpicture}};

\node at (400,0) {\begin{tikzpicture}[xscale=3*\UNIT, yscale=-4.5*\UNIT]
\begin{scope}
	\Lattice 
\end{scope} 
\begin{scope}[every node/.style={fill=white}]
	\node (a) at (0) {\begin{tikzpicture} 
	                           \node at (0,0.1) {$\emptyset$};
	                           \node at (0,.47) {\scriptsize{$\otimes$}}; 
	                           \node (z) at (0,.85) {\scriptsize{\color{blue}{2}}};
	                           \draw (z) ellipse (.35cm and .15cm);
	                           \end{tikzpicture}};
	\node (b) at (11) {\begin{tikzpicture}
	                           \node at (0,0.1) {$\emptyset$};
	                           \node at (0,.47) {\scriptsize{$\otimes$}};
	                           \node (z) at (0,.85) {\scriptsize{0}};
	                           \draw (z) ellipse (.35cm and .15cm);
	                           \end{tikzpicture}};
	\node (c) at (21) {\begin{tikzpicture}
				 \node at (0,0.1) {$\emptyset$};				  
				  \node at (0,.47) {\scriptsize{$\otimes$}};
				 \node (z) at (0,.85) {\scriptsize{\color{red}{1}}};
	                           \draw (z) ellipse (.35cm and .15cm);
	                           \end{tikzpicture}};
	\node (d) at (31) {\begin{tikzpicture}
				  \node at (0,.05) {\begin{tikzpicture}\youngDiagram{{\scriptsize{\color{blue}{2}}}}{.35}\end{tikzpicture}};
				  \node at (0,.47) {\scriptsize{$\otimes$}};
				  \node (z) at (0,.85) {\scriptsize{\color{red}{1}}};
	                           \draw (z) ellipse (.35cm and .15cm);
	                           \end{tikzpicture}};
	\node (e) at (41) {\begin{tikzpicture}
	                            \node at (0,.05) {\begin{tikzpicture}\youngDiagram{{\scriptsize{\color{blue}{2}},\scriptsize{0}}}{.35};\end{tikzpicture}};
	                            \node at (0,.47) {\scriptsize{$\otimes$}};
	                           \node (z) at (0,.85) {\scriptsize{\color{red}{1}}};
	                           \draw (z) ellipse (.35cm and .15cm);
	                           \end{tikzpicture}};
	\node (g) at (22) {\begin{tikzpicture}
	                          \node at (0,.05) {\begin{tikzpicture}\youngDiagram{{\scriptsize{\color{blue}{2}}}}{.35};\end{tikzpicture}};
				\node at (0,.47) {\scriptsize{$\otimes$}};
				\node (z) at (0,.85) {\scriptsize{0}};
	                         \draw (z) ellipse (.35cm and .15cm);
	                          \end{tikzpicture}};
	\node (i) at (32) {\begin{tikzpicture}
				\node at (0,-.15) {\begin{tikzpicture}\youngDiagram{{\scriptsize{\color{blue}{2}}},{\scriptsize{\color{red}{1}}}}{.35};\end{tikzpicture}};
				\node at (0,.47) {\scriptsize{$\otimes$}};
				\node (z) at (0,.85) {\scriptsize{0}};
	                          \draw (z) ellipse (.35cm and .15cm);
	                          \end{tikzpicture}};
	\node (k) at (42) {\begin{tikzpicture}
				\node at (0,0.05) {\begin{tikzpicture}\youngDiagram{{\scriptsize{\color{blue}{2}}}}{.35};\end{tikzpicture}};
				 \node at (0,.47) {\scriptsize{$\otimes$}};
				\node (z) at (0,.85) {\scriptsize{\color{blue}{2}}};
	                          \draw (z) ellipse (.35cm and .15cm);
	                           \end{tikzpicture}};
	\node (l) at (43) {\begin{tikzpicture}
				\node at (0,-.15) {\begin{tikzpicture}\youngDiagram{{\scriptsize{\color{blue}{2}}},{\scriptsize{\color{red}{1}}}}{.35};\end{tikzpicture}};
				\node at (0,.47) {\scriptsize{$\otimes$}};
				\node (z) at (0,.85) {\scriptsize{\color{red}{1}}};
	                           \draw (z) ellipse (.35cm and .15cm);
	                           \end{tikzpicture}};
	\node (m) at (44) {\begin{tikzpicture}
				\node at (0,.6) {};
				\node at (0,-.3) {\begin{tikzpicture}\youngDiagram{{\scriptsize{\color{blue}{2}}},{\scriptsize{\color{red}{1}}},{\scriptsize{0}}}{.35};\end{tikzpicture}};
				\node at (0,.47) {\scriptsize{$\otimes$}};
				\node (z) at (0,.85) {\scriptsize{0}};
	                           \draw (z) ellipse (.35cm and .15cm);
	                           \end{tikzpicture}};
\end{scope}
\begin{scope}[black, thick,->]
\draw (a)--(b) node[midway,left] {\scriptsize{0}};
 \draw[red] (b)--(c) node[midway,left] {\scriptsize{\color{red}{1}}};
\draw (d)--(e) node[midway,left] {\scriptsize{0}};
\draw[blue] (b)--(g) node[midway,left] {\scriptsize{\color{blue}{2}}};
\draw[red] (g)--(i) node[midway,left] {\scriptsize{\color{red}{1}}};
\draw[blue] (c)--(d) node[midway,left] {\scriptsize{\color{blue}{2}}};
\draw[blue] (d)--(k) node[midway,left] {\scriptsize{\color{blue}{2}}};  \draw[red] (i)--(l) node[midway,left] {\scriptsize{\color{red}{1}}};
\draw (i)--(m) node[midway,left] {\scriptsize{0}};
\end{scope}
\end{tikzpicture}};

\end{tikzpicture}
\caption{ \label{fig-crystal-iso}
The isomorphism $B(\Lambda_0) \simeq \Bp \otimes B(\Lambda_2)$
for $\ell = 3$.
}
\end{center}
\end{figure}

When drawing our model of $B(\Lambda_i)$,
we label each box of an $\ell$-restricted partition with 
$k \in I$, such that the main diagonal gets label $i$, 
and labels increase by $1 \bmod \ell $ as one increases
diagonals (moving right).
In this manner, the last box in the top row of $\lambda$
is labeled $k$ when
 $\crystalmap(\lambda) = $
\begin{tikzpicture}[baseline=-2pt] 
        \node at (0,0) {\;\;$k$\;\;}; 
        \draw (0,0) ellipse (.38cm and .2cm);
\end{tikzpicture}$ \; \otimes\; \mu$.
Note further that
if we have a $k$-arrow $\gamma \xrightarrow{k} \lambda$
then the box $\lambda/\gamma$ is labeled $k$ (though not necessarily
conversely). 
In fact, once one knows the structure of $\Bp$ and
the tensor product rule for crystals, one can obtain
the rule for which $k$-box $\etil{k}$ removes by
iterating $\crystalmap$.

$\Bopp$, often denoted $\Bkropp$ in the literature,
is another level $1$ perfect
crystal and is also an example of a Kirillov-Reshetikhin crystal.
$\Bopp$ is pictured in Figure \ref{fig-Bopp}; note it can be
obtained from $\Bp$ from reversing orientation of all arrows,
and we chose to relabel nodes so that still
\begin{equation*}
\ep{i}(
\begin{tikzpicture}[baseline=-2pt]
        \node at (0,0) {\;\;$k$\;\;};
        \draw (0,0) ellipse (.38cm and .2cm);
\end{tikzpicture}
) = \delta_{i,k}.
\end{equation*}

We have another crystal isomorphism
$\crystalmapopp:B(\Lambda_{i})
\xrightarrow{\simeq}
 \Bopp \otimes B(\Lambda_{i+1})$.
See Figure \ref{fig-crystal-isom-opp}.
This isomorphism is compatible with the model of $B(\Lambda_i)$
that labels nodes with $\ell$-{\em regular} partitions, that is,
those partitions $\mu$ such that the transposed diagram $\mu^T$
is $\ell$-restricted.  Then the isomorphism $\crystalmapopp$
corresponds to column
removal, in the same way $\crystalmap$ corresponds to row removal.

While the underlying $I$-colored directed graphs $B(\Lambda_i)$ are
identical, one does not obtain the $\ell$-regular model by merely
transposing the partition indexing each node of the $\ell$-restricted
model.
See Section \ref{sec-modulecrystalmodel}
for another model of $B(\Lambda_i)$ that comes from KLR algebras.

There are other level $1$ perfect crystals besides $\Bp$ and $\Bopp$,
but we do not consider them here.

\begin{figure} 
\begin{center}
\begin{tikzpicture} [scale=.5]
\draw (0,.25) ellipse (.65cm and .32cm);
\draw (3,.25) ellipse (.65cm and .32cm);
\draw (8,.25) ellipse (.65cm and .32cm);
\draw (11,.25) ellipse (.65cm and .32cm);
\draw[<-,thick] (.8,.25) -- (2.2,.25);
\draw[<-,thick] (8.8,.25) -- (10.1,.25);
\draw[<-,thick] (3.8,.25) -- (4.6,.25);
\draw[<-,thick] (6.3,.25) -- (7.2,.25);
\draw [fill] (5,0.25) circle [radius=0.03];
\draw [fill] (5.5,0.25) circle [radius=0.03];
\draw [fill] (6,0.25) circle [radius=0.03];
\node at (0,.25) {\scriptsize{1}};
\node at (3,.25) {\scriptsize{2}};
\node at (8,.25) {\scriptsize{$\ell$-1}};
\node at (11,.25) {\scriptsize{0}};
\node at (1.4,.7) {\scriptsize{1}};
\node at (4.2,.7) {\scriptsize{2}};
\node at (6.7,.7) {\scriptsize{$\ell$-2}};
\node at (9.4,.7) {\scriptsize{$\ell$-1}};
\node at (5.5,3.5) {\scriptsize{0}};
\draw[<-,thick] (10.8,.75) to [out = 130, in = 50] (0.1,.75);
\end{tikzpicture}
\end{center}
\caption{\label{fig-Bopp}
The level $1$ perfect  crystal $\Bopp$,
which is also denoted $\Bkropp$.}
\end{figure}


\begin{figure}[ht]
\begin{center}
\begin{tikzpicture} [scale=.018]
\node at (0,0) {\begin{tikzpicture}[xscale=3*\UNIT, yscale=-4.5*\UNIT]
\begin{scope}
        \LatticeHK 
\end{scope} 
\begin{scope}[every node/.style={fill=white}]
        \node (a) at (0) {$\emptyset$};
        \node (b) at (11) {\begin{tikzpicture}\youngDiagram{{\scriptsize{0}}}{.35};\end{tikzpicture}};
        \node (c) at (21) {\begin{tikzpicture}\youngDiagram{{\scriptsize{0},\scriptsize{\color{red}{1}}}}{.35};\end{tikzpicture}};
        \node (d) at (31) {\begin{tikzpicture}\youngDiagram{{\scriptsize{0},\scriptsize{\color{red}{1}},\scriptsize{\color{blue}{2}}}}{.35};\end{tikzpicture}};
        \node (e) at (41) {\begin{tikzpicture}\youngDiagram{{\scriptsize{0},\scriptsize{\color{red}{1}},\scriptsize{\color{blue}{2}},\scriptsize{0}}}{.35};\end{tikzpicture}};
        \node (g) at (22) {\begin{tikzpicture}\youngDiagram{{\scriptsize{0}},{\scriptsize{\color{blue}{2}}}}{.35};\end{tikzpicture}};
        \node (i) at (32) {\begin{tikzpicture}\youngDiagram{{\scriptsize{0},\scriptsize{\color{red}{1}}},{\scriptsize{\color{blue}{2}}}}{.35};\end{tikzpicture}};
        \node (k) at (42) {\begin{tikzpicture}\youngDiagram{{\scriptsize{0},\scriptsize{\color{red}{1}},\scriptsize{\color{blue}{2}}},{\scriptsize{\color{blue}{2}}}}{.35};\end{tikzpicture}};
        \node (l) at (43) {\begin{tikzpicture}\youngDiagram{{\scriptsize{0},\scriptsize{\color{red}{1}}},{\scriptsize{\color{blue}{2}}},{\scriptsize{\color{red}{1}}}}{.35};\end{tikzpicture}};
        \node (m) at (44) {\begin{tikzpicture}\youngDiagram{{\scriptsize{0},\scriptsize{\color{red}{1}}},{\scriptsize{\color{blue}{2}},\scriptsize{0}}}{.35};\end{tikzpicture}};
\end{scope}
\begin{scope}[black, thick,->]
\draw (a)--(b) node[midway,left] {0};
\draw[red] (b)--(c) node[midway,left] {\color{red}{1}};
\draw (d)--(e) node[midway,left] {0};
\draw[blue] (b)--(g) node[midway,left] {\color{blue}{2}};
\draw[red] (g)--(i) node[midway,left] {\color{red}{1}};
\draw[blue] (c)--(d) node[midway,left] {\color{blue}{2}};
\draw[blue] (d)--(k) node[midway,left] {\color{blue}{2}};
 \draw[red] (i)--(l) node[midway,left] {\color{red}{1}};
\draw (i)--(m) node[midway,left] {0};
\end{scope}
\end{tikzpicture}};

\node at (440,0) {\begin{tikzpicture}[xscale=3*\UNIT, yscale=-4.5*\UNIT]
\begin{scope}
        \LatticeHKWider 
\end{scope} 
\begin{scope}[every node/.style={fill=white}]
        \node (a) at (0) {\begin{tikzpicture} 
                                   \node at (0,0.1) {\scriptsize{$\otimes$}}; 
                                   \node at (.33,0.1) {$\emptyset$};
                                   \node (z) at (-.58,.1) {\scriptsize{\color{red}{1}}};
                                   \draw (z) ellipse (.35cm and .15cm);
                                   \node at (.85,0) {};
                                   \end{tikzpicture}};
        \node (b) at (11) {\begin{tikzpicture}
                                  \node at (0,0.1) {\scriptsize{$\otimes$}};
                                   \node at (.33,0.1) {$\emptyset$};
                                   \node (z) at (-.6,.1) {\scriptsize{0}};
                                   \draw (z) ellipse (.35cm and .15cm);
                                   \node at (.85,0) {};
                                   \end{tikzpicture}};
        \node (c) at (21) {\begin{tikzpicture}
        				\node at (0,0.1) {\scriptsize{$\otimes$}};
                                 \node at (.4,0.1) {\begin{tikzpicture}\youngDiagram{{\scriptsize{\color{red}{1}}}}{.35};\end{tikzpicture}};                              
                                 \node (z) at (-.55,.1) {\scriptsize{0}};
                                   \draw (z) ellipse (.35cm and .15cm);
                                   \node at (.85,0) {};
                                   \end{tikzpicture}};
        \node (d) at (31) {\begin{tikzpicture}
                                  \node at (0,0.1) {\scriptsize{$\otimes$}};
                                  \node at (.57,0.1) {\begin{tikzpicture}\youngDiagram{{\scriptsize{\color{red}{1}},\scriptsize{\color{blue}{2}}}}{.35}\end{tikzpicture}};
                                  \node (z) at (-.55,.1) {\scriptsize{0}};
                                   \draw (z) ellipse (.35cm and .15cm);
                                   \node at (1.5,0) {};
                                   \end{tikzpicture}};
        \node (e) at (41) {\begin{tikzpicture}
                                   \node at (0,0.1) {\scriptsize{$\otimes$}};
                                    \node at (.75,0.1) {\begin{tikzpicture}\youngDiagram{{\scriptsize{\color{red}{1}},\scriptsize{\color{blue}{2}},\scriptsize{0}}}{.35};\end{tikzpicture}};
                                   \node (z) at (-.55,.1) {\scriptsize{0}};
                                   \draw (z) ellipse (.35cm and .15cm);
                                   \end{tikzpicture}};
        \node (g) at (22) {\begin{tikzpicture}
                                 \node at (0,0.1) {\scriptsize{$\otimes$}};
                                 \node at (.3,0.1) {$\emptyset$};
                                \node (z) at (-.55,.1) {\scriptsize{\color{blue}{2}}};
                                 \draw (z) ellipse (.35cm and .15cm);
                                 \node at (.85,0) {};
                                  \end{tikzpicture}};
        \node (i) at (32) {\begin{tikzpicture}
                                \node at (0,0.1) {\scriptsize{$\otimes$}};
                                \node at (.39,0.1) {\begin{tikzpicture}\youngDiagram{{\scriptsize{\color{red}{1}}}}{.35};\end{tikzpicture}};
                                \node (z) at (-.55,.1) {\scriptsize{\color{blue}{2}}};
                                  \draw (z) ellipse (.35cm and .15cm);
                                  \node at (.85,0) {};
                                  \end{tikzpicture}};
        \node (k) at (42) {\begin{tikzpicture}
                                \node at (0,0.1) {\scriptsize{$\otimes$}};
                                \node at (.58,0.1)  {\begin{tikzpicture}\youngDiagram{{\scriptsize{\color{red}{1}},\scriptsize{\color{blue}{2}}}}{.35};\end{tikzpicture}};
                                \node (z) at (-.55,.1) {\scriptsize{\color{blue}{2}}};
                                  \draw (z) ellipse (.35cm and .15cm);
                                   \end{tikzpicture}};
        \node (l) at (43) {\begin{tikzpicture}
                                  \node at (0,0.1) {\scriptsize{$\otimes$}};
                                \node at (.39,0.1) {\begin{tikzpicture}\youngDiagram{{\scriptsize{\color{red}{1}}}}{.35};\end{tikzpicture}};
                                \node (z) at (-.55,.1) {\scriptsize{\color{red}{1}}};
                                   \draw (z) ellipse (.35cm and .15cm);
                                   \end{tikzpicture}};
        \node (m) at (44) {\begin{tikzpicture}
                                \node at (.35,0.1) {};
                                \node at (0,0.1){\scriptsize{$\otimes$}};
                                \node at (.4,0.1) {\begin{tikzpicture}\youngDiagram{{\scriptsize{\color{red}{1}}},{\scriptsize{0}}}{.35};\end{tikzpicture}};
                                \node (z) at (-.55,.1) {\scriptsize{\color{blue}{2}}};
                                   \draw (z) ellipse (.35cm and .15cm);
                                   \end{tikzpicture}};
\end{scope}
\begin{scope}[black,thick,->]
\draw (a)--(b) node[midway,left] {\scriptsize{0}};
 \draw[red] (b)--(c) node[midway,left] {\scriptsize{\color{red}{1}}};
\draw (d)--(e) node[midway,left] {\scriptsize{0}};
\draw[blue] (b)--(g) node[midway,left] {\scriptsize{\color{blue}{2}}};
\draw[red] (g)--(i) node[midway,left] {\scriptsize{\color{red}{1}}};
\draw[blue] (c)--(d) node[midway,left] {\scriptsize{\color{blue}{2}}};
\draw[blue] (d)--(k) node[midway,left] {\scriptsize{\color{blue}{2}}};  \draw[red] (i)--(l) node[midway,left] {\scriptsize{\color{red}{1}}};
\draw (i)--(m) node[midway,left] {\scriptsize{0}};
\end{scope}
\end{tikzpicture}};


\end{tikzpicture}
\caption{ \label{fig-crystal-isom-opp} 
The isomorphism $B(\Lambda_0) \simeq \Bopp \otimes B(\Lambda_1)$
for $\ell = 3$.
}
\end{center}
\end{figure}


\section{Definition of the KLR algebra $R(\nu)$ and some functors}
\label{sec-KLR}
 
In what follows we let $[k]$ be the quantum integer in the
indeterminant $q$,
\begin{equation}
[k] = q^{k-1} + q^{k-3} + \dots + q^{1-k} \quad \text{and} \quad [k]! = [k] [k-1] \dots [1].
\end{equation}
For $\nu = \sum_{i \in I} \nu_i \alpha_i$ in $Q^+$ with $|\nu| = m$, we define $\seq(\nu)$ to be all sequences 
\begin{equation*}
\und{i} = (i_{1}, i_{2}, \dots, i_{m})
\end{equation*}
such that $i_k$ appears $\nu_k$ times. For $\und{i} \in \seq(\nu)$ and $\und{j} \in \seq(\mu)$, $\und{ij}$, will denote the concatenation of the two sequences unless otherwise specified. It follows that $\und{ij} \in \seq(\nu + \mu)$. We write 
\begin{equation}
i^n = (\underbrace{i,i, \dots ,i}_{n}).
\end{equation}
There is a left  action of the symmetric group, $\Sy{m}$, on $\seq(\nu)$ defined by,
\begin{equation}
s_{k}(\und{i}) = s_k\Big( i_{1}, i_{2}, \dots ,i_{k}, i_{k+1}, \dots ,i_m\Big) = (i_{1}, i_{2}, \dots, i_{k+1}, i_{k}, \dots, i_{m})
\end{equation}
where $s_k$ is the adjacent transposition in $\Sy{m}$ that interchanges $k$ and $k+1$. 

Since this paper only considers KLR algebras of type
$A^{(1)}_{\ell-1}$, we simplify the definition below from that for
general type. The definition for arbitrary symmetrizable types can be
found in \cite{KL09}, \cite{KL11}, and \cite{Rou08}.
Using the more general definition with Rouquier's parameters
$Q_{i,j}(u,v)$ will not change the results or proofs in this paper, as
they concern crystal-theoretic phenomena.
 There is also a
diagrammatic presentation of KLR algebras which can be found in
\cite{KL09}, \cite{KL11}. By results of Brundan-Kleshchev \cite{BK09a},
\cite{BK09b}, there is an isomorphism between $R^{\Lambda}(\nu)$ and
$H^{\Lambda}_{\nu}$ where $H^{\Lambda}_\nu$ is a block of the
cyclotomic Hecke algebra $H^{\Lambda}_m$ as defined in
\cite{AK94, BM94, Che87}. Hence readers unfamiliar with KLR algebras
can translate all statements and proofs in terms of Hecke algebras
throughout the paper. We remark that
historically, this is the original setting in
which the theorems from this paper were proved \cite{V03, Vnotes}.
In fact the
reader can think of all results as being stated for $\MB{F}_\ell
\Sy{m}$ in the case that $\ell$ is prime if the other algebras are
not familiar.

For $\nu \in Q^+$ with $|\nu| = m$, the {\emph{KLR algebra}}
$R(\nu)$ is the associative, graded, unital $\MB{C}$-algebra generated
by \begin{equation}
1_{\und{i}} \;\; \text{for} \;\; \und{i} \in \seq(\nu),
\quad x_r \;\; \text{for} \;\; 1 \leq r \leq m,
\quad  \psi_r \;\; \text{for} \;\; 1 \leq r \leq m-1,
\end{equation}
subject to the following relations, where $\und{i}, \und{j} \in
\seq(\nu)$ and equality between $i_r$ and $i_t$ is taken to mean
equality in $\MB{Z}/\ell \MB{Z}$.
\begin{equation}
\label{eq-idem}
1_{\und{i}}1_{\und{j}} = \delta_{\und{i},\und{j}}1_{\und{i}}, \quad\quad x_r1_{\und{i}} = 1_{\und{i}}x_r, \quad\quad \psi_r 1_{\und{i}} = 1_{s_r(\und{i})}\psi_r,  \quad\quad x_rx_t = x_tx_r,
\end{equation}
\begin{align}
 \psi_r\psi_t &= \psi_t \psi_r \qquad   \text{if} \; |r-t| > 1,  \\
\label{square-relation} \psi_r\psi_r1_{\und{i}} &= 
   \begin{cases}
    0 & \text{if $i_r = i_{r+1}$}\\
    (x_r^{-a_{i_r i_{r+1}}}+ x_{r+1}^{-a_{i_{r+1}i_r}})1_{\und{i}} & 
		\text{if $i_r = i_{r+1} \pm 1$}\\
    1_{\und{i}} & \quad \quad \text{otherwise,}
  \end{cases} 
\\
  \label{braid-relation}
(\psi_r \psi_{r+1} \psi_r - \psi_{r+1}\psi_r \psi_{r+1})1_{\und{i}}
 &=  
   \begin{cases}
    1_{\und{i}}  & \text{if $\ell > 2$ and $i_r = i_{r+2} = i_{r+1} \pm 1$}\\
    (x_r + x_{r+2})1_{\und{i}} &  \text{if $\ell =2$ and
$i_r = i_{r+2} = i_{r+1} \pm 1$}\\
    0 & \text{otherwise,}
  \end{cases}    \\ 
 \label{dot-past-crossing} (\psi_rx_t - x_{s_r(t)}\psi_r)1_{\und{i}}&=
   \begin{cases}
    1_{\und{i}} & \quad \quad \text{if $t = r$ and $i_r = i_{r+1}$}\\
    -1_{\und{i}} & \quad \quad \text{if $t = r+1$ and $i_r = i_{r+1}$}\\
    0 & \quad \quad \text{otherwise.}
  \end{cases} 
\end{align}
The elements $1_{\und{i}}$ are idempotents in $R(\nu)$
by \eqref{eq-idem}
and the identity element is given by
\begin{equation}
1_\nu = \sum_{\und{i} \in \seq(\nu)} 1_{\und{i}}.
\end{equation}
Thus, as a vector space $R(\nu)$ decomposes as, 
\begin{equation}
R(\nu) = \bigoplus_{\und{i},\und{j} \in \seq(\nu)} 1_{\und{i}} R(\nu) 1_{\und{j}}.
\end{equation}
The generators of $R(\nu)$ are graded as, 
\begin{equation}
\deg(1_{\und{i}}) = 0, \;\;\; \deg(x_r1_{\und{i}}) = 2, \;\;\; \deg(\psi_r1_{\und{i}}) = -(\alpha_{i_r}, \alpha_{i_{r+1}}).
\end{equation}
We define
\begin{equation}
R = \bigoplus_{\nu \in Q^+}R(\nu).
\end{equation}
Notice that while $R(\nu)$ is unital, $R$ is not. 

For each $w \in \Sy{m}$ we fix once and for all a reduced expression
\begin{equation}
\wh{w} = s_{i_1}s_{i_2} \dots s_{i_t}.
\end{equation} 
Observe that the $s_k$ are Coxeter generators of $\Sy{m}$, and $t$ is
the Coxeter length of $w$. Let $\psi_{\wh{w}} = \psi_{i_1} \psi_{i_2}
\dots \psi_{i_t}$ correspond to the chosen reduced expression $\wh{w}$.
For $\und{i},\und{j} \in \seq(\nu)$, let $_{\und{j}}\Sy{\und{i}}$ be
the permutations in  $\Sy{m}$ that take $\und{i}$ to $\und{j}$. 

\begin{theorem} \label{basis-theorem} \cite[Theorem 2.5]{KL09} 
As a $\mathbb{C}$-vector space $1_{\und{j}}R(\nu)1_{\und{i}}$ has basis,
\begin{equation}
\{\psi_{\wh{w}} x_1^{b_1} \dots x_{m}^{b_m} 1_{\und{i}} \; | \; w \in {}_{\und{j}}{\Sy{}}_{\und{i}}, \; b_r \in \mathbb{Z}_{\geq 0}\}.
\end{equation}
\end{theorem}

It is known that all simple $R(\nu)$-modules are finite dimensional \cite{KL09}. For this reason, in this paper we only consider the category of finite-dimensional KLR-modules $R(\nu) \Mod$ and $R \Mod$.

We often refer to $1_{\und{i}}M$ as the $\und{i}$-{\em weight} space
of $M$ and any $0 \neq v \in 1_{\und{i}}M$ as a {\em weight vector}.
A {\em weight basis} is a basis consisting of weight vectors.

We define the graded character of an $R(\nu)$-module to be
\begin{equation}
\Char(M) = \sum_{\und{i} \in \seq(\nu)} \text{gdim}(1_{\und{i}}M) \cdot [\und{i}].
\end{equation}
Here $\text{gdim}(1_{\und{i}}M)$ is an element of
$\MB{Z}[q,q^{-1}]$, and hence $\Char(M)$ is an element of the free
$\MB{Z}[q,q^{-1}]$-module generated by all $[\und{i}]$ for $\und{i} \in
\seq(\nu)$. We will let $\ch(M)$ denote the multiset that is the
support of $\Char(M)$ so that 
\begin{equation}
\Char(M)|_{q=1} = \sum_{[\und{i}] \in \ch(M)} [\und{i}].
\end{equation}
Our notational convention
is to write $\und{i} \in \seq(\nu)$ but write $[\und{i}] \in
\ch(M)$. Since characters are an important combinatorial tool, it is
worthwhile to set a special notation for them.

Because $R$ is a graded algebra, we will only work with homomorphisms between $R$-modules that are either degree preserving or degree homogeneous. We denote the $\mathbb{C}$-vector space of degree preserving homomorphisms between $R(\nu)$-modules $M$ and $N$ by $\Hom(M,N)$. Since any homogeneous homomorphism can be interpreted as degree preserving by shifting the grading on our target or source module, then we can write the $\mathbb{C}$-vector space of homogeneous homomorphisms between $M$ and $N$, $\HOM(M,N)$, by
\begin{equation}
\HOM(M,N) = \bigoplus_{k \in \mathbb{Z}} \Hom(M,N\{k\}).
\end{equation}
While the grading is important, it was shown in \cite{KL09} that there is a unique grading on a simple $R$-module up to overall grading shift. Since this paper concerns simple modules, we will rarely use or discuss the grading.
All isomorphisms between modules will be taken up to overall
grading shift.

\begin{remark} \label{x-nilpotent}
Because $x_r1_{\und{i}} \in R(\nu)$ is always positively graded for $1 \leq r \leq |\nu|$ and $\und{i} \in \seq(\nu)$, then on a finite dimensional $R(\nu)$-module, $M$, $x_r1_{\und{i}}$ will always act nilpotently.
\end{remark}

\subsection{Trivial and sign modules}

For $k \in \N, i \in I$, define positive roots 
\begin{equation}\gammaplus{i}{k} := \alpha_i + \alpha_{i+1} + \dots + \alpha_{i+k-1} \quad \text{and} \quad \gammaminus{i}{k} := \alpha_i + \alpha_{i-1} + \dots + \alpha_{i-k+1}
\end{equation} 
of height $k$.
As in Remark \ref{rem-k} we interpret subscripts to be in $I$.
Because 1-dimensional modules will play a key role in our main
theorems below, we give the following classification.

\begin{proposition} \label{class-of-triv}
If $M$ is a 1-dimensional $R(\nu)$-module with
$|\nu| = m$, then $M$ has character
\begin{equation} \label{ascending-triv-form}
\Char(M) = [i, i+1, \dots,  i+m-2, i+m-1]
\end{equation}
or 
\begin{equation} \label{descending-triv-form}
\Char(M) = [i,i-1,\dots , i-m+2, i-m+1]
\end{equation}
and $\nu = \gammaplus{i}{m}$ or $\nu = \gammaminus{i}{m}$ respectively. The entries in $\Char(M)$ should be taken modulo $\ell$.
\end{proposition}

\begin{proof}
Our proof uses similar techniques to those used in \cite{KR10} for the
study of calibrated (or homogeneous) modules, of which the
1-dimensional $R(\nu)$-modules form a subset.
Let $M$ be spanned by the vector $v$,
and let $\und{i}$ be the unique element of $\seq(\nu)$ such that $1_{\und{i}}M \neq \zero$. We write
\begin{equation}
\und{i} = (i_1,i_2, \dots ,i_m).
\end{equation}
Recall $x_r1_{\und{i}}$, for $1 \leq r \leq m$ acts nilpotently on $M$ by Remark \ref{x-nilpotent}. Since $M$ is 1-dimensional, then $x_r1_{\und{i}}v = 0$.

Suppose $i_{r} = i_{r+1}$ (recall that we interpret $i_r, i_{r+1} \in \mathbb{Z}/\ell \mathbb{Z}$). Then from \\relation \eqref{dot-past-crossing},
\begin{equation}
(\psi_rx_r - x_{r+1}\psi_r)1_{\und{i}}v = 1_{\und{i}}v = v,
\end{equation}
but this is impossible as $x_r$ and $x_{r+1}$ both act as zero. Thus $i_r \neq i_{r+1}$.

Let $\psi_r1_{\und{i}}v = a_rv$ for some constant $a_r$. Given
we showed 
$i_r \neq i_{r+1}$, if additionally we suppose
 $i_{r+1} \neq \pm 1 + i_r$, then by relation
\eqref{square-relation}, \begin{equation}
a_r^2v = \psi_r^21_{\und{i}}v = v,
\end{equation}
so $a_r \neq 0$. Then,
$0 \neq \psi_r 1_{\und{i}}v = 1_{s_r(\und{i})} \psi_r v $
$ \in 1_{s_r(\und{i})}M$. 
But
$s_r(\und{i}) \neq \und{i}$
given $i_r \neq i_{r+1}$
which contradicts the fact that $M$ is 1-dimensional.
Hence $i_{r+1} = \pm 1 + i_r$ and \begin{equation}
a_r^2v = \psi_r^2 1_{\und{i}}v = (x_r + x_{r+1})1_{\und{i}}v = 0
\end{equation}
showing $a_r = 0$. Thus $\psi_r1_{\und{i}}v = 0$ for all $r$.

In the case when $\ell = 2$, $i_{r+1} = \pm 1 + i_r$ fully determines
$\und{i}$ and agrees with the conclusions of the proposition,
 so for the rest of the proof we assume $\ell > 2$. Suppose
$i_r = i_{r+2}$ for some $1 \leq r \leq m-2$. Since $i_{r+1} = \pm 1 +
i_r = \pm 1 + i_{r+2}$, relation \eqref{braid-relation} gives
\begin{equation} \label{braid-on-1-dim}
0 = (a_r^2a_{r+1} - a_ra^2_{r+1})1_{\und{i}}v = (\psi_r\psi_{r+1}\psi_r - \psi_{r+1}\psi_r\psi_{r+1})1_{\und{i}}v = 1_{\und{i}}v = v
\end{equation}
which is a contradiction as $a_r=0$ but $v \neq 0$.
So $i_r \neq i_{r+2}$ and $\und{i}$ has form
\eqref{ascending-triv-form} or \eqref{descending-triv-form}. In
particular $\nu = \gamma^{\pm}_{i;m}$. To show that such $M$ actually
exist one need only check that setting $1_{\und{i}}v = v$ for $\und{i}$
as in \eqref{ascending-triv-form} or \eqref{descending-triv-form}, and
$x_rv = \psi_rv = 1_{\und{j}}v = 0$ for $\und{j} \neq \und{i}$,
satisfies all of
the relations on the generators of
$R(\gamma^{\pm}_{i;m})$.

\end{proof}

When $k > 0$, we denote the 1-dimensional $R(\gammaplus{i}{k})$-module $T$ with ascending character,
\begin{equation}
\Char \Ti = [i, i+1, \dots ,i + k-1] \quad \text{as} \quad \Ti =
\Tii{i}{k} = \T{i, i+1, \dots ,i + k-1},
\end{equation}
and the 1-dimensional $R(\gammaminus{i}{k})$-module $S$ with descending character,
\begin{equation}
\Char S = [i, i-1, \dots, i-k+1] \quad \text{as} \quad S = \Sii{i}{k} = \Si{i,i-1, \dots, i-k+1}.
\end{equation}
We refer to $k$ as the {\emph{height}} of $\Tii{i}{k}$ and $\Sii{i}{k}$
respectively. When $k = 0$, then $\gamma^{\pm}_{i;k} = 0$ and
$\Tii{i}{k} = \UnitModule$, the unique simple $R(0)$-module, which we
will refer to as the \emph{unit} module.
In this paper we choose to work with 1-dimensional modules with
ascending character.
(These are analogous to trivial modules for the
affine Hecke algebra or symmetric group.) One could also have chosen to
use the 1-dimensional modules with descending character (analogous to
sign modules) with obvious modifications.  Hence we will informally
refer to each type of module as a trivial or sign module, respectively.
\begin{example}
\label{exl-onedim}
In type $A^{(1)}_3$, with
\begin{equation*}
\nu = 2\alpha_0 + 2\alpha_1 + \alpha_2 + \alpha_3,
\end{equation*}
of height $6$,
$R(\nu)$ has a simple 1-dimensional module $T = \Tii{0}{6}$ with
\begin{equation*}
\Char(T) = [0,1,2,3,0,1],
\end{equation*}
but it is convenient to also write this as
\begin{equation*}
\Char(T) = [0,1,2,3,4,5].
\end{equation*}
\end{example}



\subsection{Induction  and restriction}
\label{sec-ind}

It was shown in \cite{KL09} and \cite{KL11} that for $\nu, \mu \in Q^+$ there is a non-unital embedding
\begin{equation}
R(\nu) \otimes R(\mu) \hookrightarrow R(\nu + \mu).
\end{equation}
This map sends the idempotent $1_{\und{i}} \otimes 1_{\und{j}}$ to $1_{\und{ij}}$. The identity $1_{\nu} \otimes 1_{\mu}$ of $R(\nu) \otimes R(\mu)$ has as its image
\begin{equation}
\sum_{\und{i} \in \seq(\nu)} \sum_{\und{j} \in \seq(\mu)} 1_{\und{ij}}.
\end{equation}
Using this embedding one can define induction and restriction functors, 
\begin{align}
\ind_{\nu, \mu}^{\nu + \mu}: (R(\nu) \otimes R(\mu)) \Mod & \rightarrow R(\nu + \mu) \Mod\\
\notag
M & \mapsto R(\nu + \mu) \otimes_{R(\nu) \otimes R(\mu)} M
\end{align}
and
\begin{equation}
\res_{\nu, \mu}^{\nu + \mu}: R(\nu + \mu) \Mod \rightarrow (R(\nu) \otimes R(\mu)) \Mod.
\end{equation}
In the future we will write $\ind_{\nu,\mu}^{\nu+\mu} = \ind$ and $\res_{\nu,\mu}^{\nu+\mu} = \res$ when the algebras are understood from the context. More generally we can extend this embedding to finite tensor products 
\begin{equation}
R(\nu^{(1)}) \otimes R(\nu^{(2)}) \otimes \dots \otimes R(\nu^{(k)}) \hookrightarrow R(\nu^{(1)} + \nu^{(2)} + \dots + \nu^{(k)}).
\end{equation}
We refer to the image of this embedding as a \emph{parabolic subalgebra} and denote it by $R(\und{\nu}) \subset R(\nu^{(1)} + \dots + \nu^{(k)})$. We denote the image of the identity under this embedding as $1_{\und{\nu}}$.
It follows from Theorem \ref{basis-theorem} that $R(\nu^{(1)} +
\nu^{(2)} + \dots + \nu^{(k)})1_{\und{\nu}}$ is a free right
$R(\und{\nu})$-module and
$1_{\und{\nu}}R(\nu^{(1)} + \nu^{(2)} + \dots + \nu^{(k)})$
is a free left $R(\und{\nu})$-module.
 Let $m_i = |\nu^{(i)}|$ and set
\begin{equation}
\label{eq-parabolic}
P = (m_1, \dots, m_k) \quad \text{and} \quad \Sy{P} = \Sy{m_1} \times
\Sy{m_1} \times \dots \times \Sy{m_k}.
 \end{equation}
Let $\Sy{m_1 + \dots + m_k} / \Sy{P}$ be the collection of minimal
length left coset representatives of $\Sy{P}$ in $\Sy{m_1 + \dots +
m_k}$ and $\Sy{P}\backslash \Sy{m_1 + \dots + m_k}$ be the collection
of minimal length right coset representatives of $\Sy{P}$ in $\Sy{m_1 +
\dots + m_k}$.
We construct a weight basis for an induced module as follows. If $M$ is
an $R(\und{\nu})$-module with weight basis $U$ then
$\ind_{\und{\nu}}^{\nu^{(1)} + \dots + \nu^{(k)}} M$ has weight basis
\begin{equation} \label{ind-basis}
\{\psi_{\wh{w}} \otimes u  \; | \; w \in \Sy{m_1 + \dots + m_k} /
\Sy{P}, u \in U \}. 
\end{equation}
Induction is left adjoint to restriction (a property known as
Frobenius reciprocity),
\begin{equation}
\HOM_{R(\nu^{(1)} + \dots + \nu^{(k)})}(\ind_{\und{\nu}}^{\nu^{(1)} + \dots + \nu^{(k)}} M, N) \cong \HOM_{R(\und{\nu})}(M,\res_{\und{\nu}}^{\nu^{(1)} + \dots + \nu^{(k)}}N).
\end{equation}

Given $\und{i} \in \seq(\nu)$ and $\und{j} \in \seq(\mu)$, a shuffle of
$\und{i}$ and $\und{j}$ is an element $\und{k}$ of $\seq(\nu + \mu)$
such that $\und{k}$ has $\und{i}$ as a subsequence and $\und{j}$ as the
complementary subsequence.
We denote by $\und{i} \shuffle
\und{j}$ the formal sum of all shuffles of $\und{i}$ and $\und{j}$.
The multi-set of all shuffles of $\und{i}$ and
$\und{j}$ are in bijection with the minimal length left coset
representatives  $\Sy{|\nu| + |\mu|}/\Sy{|\nu|} \times \Sy{|\mu|}$.
Using the definition of degree from KLR algebras, we can associate to
any shuffle a degree which we denote as
$\deg(\und{i},\und{j},\und{k})$.
Then the quantum shuffle of $\und{i}$ and $\und{j}$ is
\begin{equation}
\und{i} \Shuffle \und{j} =
\sum_{\sigma \in \Sy{|\nu| + |\mu|}/\Sy{|\nu|} \times \Sy{|\mu|}}
q^{\deg(\und{i},\und{j},\sigma(\und{ij}))} \sigma(\und{ij}),
\end{equation}
so that $\und{i} \shuffle \und{j} = (\und{i} \Shuffle
\und{j})|_{q=1}$. Note that we will usually shuffle characters,
hence we also write $[\und{i}] \Shuffle [\und{j}]$. For an
$R(\mu)$-module $M$ and $R(\nu)$-module $N$ it was shown in
\cite{KL09} that \begin{equation}
\Char(\ind_{\mu,\nu}^{\mu + \nu} M \boxtimes N) = \Char(M) \Shuffle \Char(N).
\end{equation}
This identity is referred to as the {\bf Shuffle Lemma}.


\subsection{Simple modules of $R(n\alpha_i)$}
\label{R(kalpha)-modules}

For $\nu = n\alpha_i$, induction allows for a particularly easy
description of all simple $R(n\alpha_i)$-modules. Let $\Lii{i}$ be
the 1-dimensional $R(\alpha_i)$-module.
(Note $x_1 1_{i}$ acts as zero.)
Then the unique simple $R(n\alpha_i)$ module is 
\begin{equation}
\Lii{i^n} := \ind_{\alpha_i, \alpha_i,  \dots, \alpha_i}^{n\alpha_i} \Lii{i} \boxtimes \dots \boxtimes \Lii{i}
\end{equation}
up to overall grading shift, which we may shift to have character
\begin{equation}
\Char(\Lii{i^n}) = [n]! [i,i,\dots, i].
\end{equation}


\subsection{Crystal operators on the category ${R}\Mod$}

In the previous section we defined induction and restriction for KLR
algebras. Following the work of Grojnowski \cite{Gro99} where crystal
operators were developed as functors on the category of modules
over affine Hecke algebras of type $A$ (or \cite{Klesh} for
$\MB{F}_\ell \Sy{m}$), the KLR analogues of crystal operators were
introduced in \cite{KL09}, and further developed in \cite{LV11},
\cite{KK12}. For each $i \in I$, if $M \in R(\nu)\Mod$ 
and  $\nu -\alpha_i \in Q^+$,
 define the functor $\Delta_i:
R(\nu)\Mod \rightarrow R(\nu - \alpha_i) \otimes R(\alpha_i) \Mod$ as
the restriction
\begin{equation}
\Delta_i M := \res^{\nu}_{\nu-\alpha_i,\alpha_i}M.
\end{equation}
Note that this is equivalent to multiplying $M$ by $1_{\nu-\alpha_i}
\otimes 1_{\alpha_i}$. It is also sometimes useful to think of this
functor as killing all weight spaces corresponding to elements of
$\seq(\nu)$ that do not end in $i$. 
If $\nu -\alpha_i \not\in Q^+$ then $\Delta_i M = \zero$.
 We similarly define
\begin{equation}
\Delta_{i^n} M := \res_{\nu-n\alpha_i,n\alpha_i}^\nu M.
\end{equation}
Next define the functor $\e{i}: R(\nu) \Mod \rightarrow R(\nu-\alpha_i) \Mod$ as the restriction,
\begin{equation}
\e{i}M := \res^{R(\nu-\alpha_i)\otimes R(\alpha_i)}_{R(\nu-\alpha_i)}  \Delta_i M
\end{equation}
When $M$ is simple, we can further refine this functor by setting
\begin{equation}
\etil{i}M := \soc \e{i}M.
\end{equation}
We measure how many times we can apply $\etil{i}$ to a simple module $M$ by
\begin{equation}
\ep{i}(M) := \max\{ n \geq 0 \; | \; (\etil{i})^n M \neq \zero \;\}.
\end{equation}
Let $\ftil{i}:R(\nu) \Mod \rightarrow R(\nu+\alpha_i) \Mod$ be defined by
\begin{equation}
\ftil{i}M := \cosoc \ind M \boxtimes L(i).
\end{equation}
We also set 
$\wt(M) = - \nu$ if $M \in R(\nu)\Mod$,
and $\p_i(M) = \ep{i}(M) -\langle h_i, \nu \rangle$.
This data is all part of a crystal datum that defines
the structure of the crystal graph $B(\infty)$ on the 
simple $R$-modules.  See Section \ref{sec-modulecrystalmodel}.

Some of the most important facts about $\e{i}, \etil, \ftil{i}$ 
stated in \cite{KL09} are given in the following proposition.

\begin{proposition} \label{crystal-op-facts} \mbox{}
Let $i \in I$, $\nu \in Q^+$, $n \in \Z_{>0}$.
\begin{enumerate}
\item
Let $M \in R(\nu)\Mod$. Then
\begin{equation*}
\Char(\Delta_{i^n}M) = \sum_{\und{j} \in \seq(\nu - n\alpha_i)} \text{gdim}(1_{\und{j}i^n}M) \cdot \und{j}i^n,
\end{equation*}

\item \label{ftil-and-ep} Let $N \in R(\nu) \Mod$ be irreducible and $M = \ind_{\nu,n\alpha_i}^{\nu + n\alpha_i}N \boxtimes \Lii{i^n}$. Let $\ep{} = \ep{i}(N).$ Then
\begin{enumerate}
\item \label{pull-off-i's} $\Delta_{i^{\ep{} +n}}M \cong (\etil{i})^\ep{}N \boxtimes \Lii{i^{\ep{}+n}}$.
\item \label{cosoc-is-simple} $\cosoc M$ is irreducible, and $\cosoc M \cong (\ftil{i})^n N$, $\Delta_{i^{\ep{}+n}}(\ftil{i})^n N \cong (\etil{i})^{\ep{}}N \boxtimes \Lii{i^{\ep{}+n}}$, and $\ep{i}((\ftil{i})^n N) = \ep{} + n$.
\item $(\ftil{i})^n N$ occurs with multiplicity one as a composition factor of $M$.
\item \label{all-comp-have-less-i} All other composition factors $K$ of $M$ have $\ep{i}(K) < \ep{}+n$.
\end{enumerate}

\item Let
$\und{\mu} = ({\mu_1} \alpha_i, \dots, {\mu_r} \alpha_i)$
with $\sum^r_{k=1} \mu_k = n$.
\begin{enumerate}
\omitt{\item The module $\Lii{i^n}$ over the algebra $R(n\alpha_i)$ is the only graded irreducible module, up to isomorphism.}
\item All composition factors of $\res_{\und{\mu}}^{n\alpha_i}\Lii{i^n}$ are isomorphic to $\Lii{i^{\mu_1}} \boxtimes \dots \boxtimes \Lii{i^{\mu_r}}$, and $\soc(\res_{\und{\mu}}^{n \alpha_i} \Lii{i^n})$ is irreducible.
\item $\etil{i} \Lii{i^n} \cong \Lii{i^{n-1}}.$
\end{enumerate}

\item \label{etil-is-only-special-compfactor} Let $M \in R(\nu) \Mod$ be irreducible with $\ep{i}(M) > 0$. Then $\etil{i}M = \soc(\e{i}M)$ is irreducible and $\ep{i}(\etil{i}M) = \ep{i}(M) - 1$. Furthermore if $K$ is a composition factor of $\e{i}M$ and $K \not\cong \etil{i}M$,
then $\ep{i}(K) < \ep{i}(M) -1$.

\item \label{e-etil-relationship}
For irreducible $M \in R(\nu) \Mod$ let $m = \ep{i}(M)$. Then
$\e{i}^m M$ is isomorphic to $(\etil{i})^m M^{\oplus [m]!}$.
In particular, if $m =1$ then $\e{i}M = \etil{i}M$.

\item \label{e-and-f-undo-eachother} For irreducible modules $N \in R(\nu) \Mod$ and $M \in R(\nu+\alpha_i) \Mod$ we have $\ftil{i}N \cong M$ if and only if $N \cong \etil{i}M$.

\item Let $M, N \in R(\nu) \Mod$ be irreducible. Then $\ftil{i}M \cong \ftil{i}N$ if and only if $M \cong N$. Assuming $\ep{i}(M), \ep{i}(N) > 0$, $\etil{i}M \cong \etil{i}N$ if and only if $M \cong N$.

\end{enumerate}

\end{proposition}

On the level of characters, $\e{i}$ roughly removes an $i$ from the rightmost entry of a module's character. We can construct analogous functors for removal of $i$ from the left side of a module's character, as well as an analogue to $\ftil{i}$. These are denoted by $\ech{i}$, $\etilch{i}$, $\ftilch{i}$ and we will use them extensively in this paper. We use the involution $\sigma$ introduced below to define them. Let $w_0$ be the longest element of $\Sy{|\nu|}$. Then $\sigma: R(\nu) \rightarrow R(\nu)$ is defined as follows:
\begin{align}
1_{\und{i}} &\mapsto 1_{w_0(\und{i})}
\\
x_r &\mapsto x_{|\nu|+1-r}
\\
\psi_r 1_{\und{i}} &\mapsto
(-1)^{\delta_{i_r,i_{r+1}}}\psi_{|\nu|-r}1_{w_0(\und{i})}.
\end{align}
For an $R(\nu)$-module $M$, let $\sigma^*M$ be the $R(\nu)$-module $M$ but with the action of $R(\nu)$ twisted by $\sigma$,
\begin{equation*}
r \cdot u = \sigma(r)u.
\end{equation*}

Now let $\ech{i}: R(\nu) \Mod \rightarrow R(\nu-\alpha_i) \Mod$ be the restriction functor defined as
\begin{align}
&\ech{i} := \sigma^*\e{i} \sigma = \res^{R(\alpha_i)\otimes R(\nu-\alpha_i)}_{R(\nu - \alpha_i)} \circ \res_{\alpha_i,\nu-\alpha_i}^\nu,
\intertext{
and similarly,}
&\etilch{i}M := \sigma^*(\etil{i}(\sigma^*M)) = \soc \ech{i}M,
\\
&\ftilch{i}M := \sigma^*(\ftil{i}(\sigma^*M)) = \cosoc \ind^{\nu+\alpha_i}_{\alpha_i,\nu} \Lii{i} \boxtimes M,
\\
&\epch{i}(M) := \ep{i}(\sigma^*M) = \max\{n \geq 0 \; | \; (\etilch{i})^n M \neq \zero \}.
\end{align}

Note that by the exactness of restriction,
$\e{i}, \ech{i}$ are exact functors, while $\etil{i}$ and
$\etilch{i}$ are only left exact, and $\ftil{i}$ and $\ftilch{i}$
are only right exact.
When $k \in \N$, the indices on $\e{k}$, $\etil{k}$, $\ech{k}$,
$\etilch{k}$, $\ftil{k}$, $\ftilch{k}$ should always be interpreted
modulo $\ell$, i.e.~we also identify $k \in I$.

\begin{example}
\label{exl-triv}
The module 
of Example
\ref{exl-onedim}
can be constructed as
$\Tii{0}{6}= 
\ftil{5} \ftil{4} \ftil{3} \ftil{2} \ftil{1} \ftil{0} \UnitModule$
or as  
$\ftil{5} \ftil{4} \ftil{3} \ftil{2} \ftil{1} L(0)$.
Since $\ell=4$ this is also
$\ftil{1} \ftil{0} \ftil{3} \ftil{2} \ftil{1} \ftil{0} \UnitModule$,
but for the purposes of this paper, we prefer the first expression.

\end{example}
\begin{remark} \label{rem-char-epsilon}
There is a nice character-theoretic interpretation of $\ep{i}$ and $\epch{i}$. Let $M$ be a simple $R(\nu)$-module with $|\nu| = m$. Then
\begin{enumerate}[a.)]
\item $\ep{i}(M) = c$ implies that there exists 
\begin{equation*}
\und{i} = (i_1, \dots, i_{m-c},\underbrace{i,i, \dots ,i}_{c})
\end{equation*}
such that $1_{\und{i}}A \neq \zero$.
In other words $[\und{i}]$ is in the support of $M$; however no $[\und{j}]$ such that 
\begin{equation*}
\und{j} = (i_{1}, \dots, i_{m-c-1}, \underbrace{i ,i ,\dots, i}_{c+1}).
\end{equation*}
is in $\supp{M}$.

\item $\epch{i}(M) = c$ implies that there exists $[\und{i}]$ in the support of $M$ of the form
\begin{equation*}
\und{i} = (\underbrace{i,i,\dots ,i}_{c},i_{c+1}, \dots ,i_m)
\end{equation*}
but no $[\und{j}]$ of the form
\begin{equation*}
\und{j} = (\underbrace{i,i,\dots ,i}_{c+1},i_{c+2}, \dots ,i_m).
\end{equation*}

\end{enumerate}
\end{remark}


\subsection{Serre relations}
\label{sec-serre}

Because the functors $\e{i}$, $i \in I$, are exact, they descend to well-defined linear operators on the Grothendieck group of $R$, $G_0(R)$. It is shown in \cite{KL09,KL11} that these operators satisfy the quantum Serre relations, and that these relations are in fact minimal. We have
\begin{equation}
\sum^{-a_{ij}+1}_{r = 0} (-1)^r \e{i}^{(-a_{ij}+1-r)}\e{j}\e{i}^{(r)}[M] = \zero.
\end{equation}
for all $i \neq j \in I$  and $M \in R\Mod$,
where $e^{(r)}_i = \frac{1}{ [r] !}e^r_i$ is the divided power.
(Recall $a_{ij} = \langle h_i, \alpha_j \rangle$.)
 The minimality of these relations imply that, for $0 \leq c < -a_{ij}+1$, 
\begin{equation} \label{Serre-relations}
\sum^c_{r=0} (-1)^r\e{i}^{(c-r)}e_je_i^{(r)}
\end{equation}
is never the zero operator on $G_0(R)$ by the quantum Gabber-Kac Theorem \cite{Lus10} and the work of \cite{KL09,KL11}, which essentially computes the kernel of the map from the free algebra on generators $\e{i}$ to $G_0(R)$.


\subsection{Jump}
\label{sec-jump}

When we apply $\ftil{i}$ to irreducible $R(\nu)$-module $M$ for $i \in I$, then Proposition \ref{crystal-op-facts}.\ref{ftil-and-ep} tells us that $\ftil{i}M$ is an irreducible $R(\nu+\alpha_i)$-module with
\begin{equation}
\ep{i}(\ftil{i}M) = \ep{i}(M) + 1.
\end{equation}
We could also ask whether $\epch{i}(\ftil{i}M)$ and $\epch{i}(M)$ differ. Questions like this motivate the introduction of the
function $\jump{i}$,
which is based on a concept for Hecke algebras in \cite{Gro99}, and was introduced for KLR algebras and studied extensively in \cite{LV11}.

\begin{definition}
Let $M$ be a simple $R(\nu)$-module, and let $i \in I$. Then
\begin{equation}
\jump{i}(M) := \max \{ J \geq 0 \; | \; \epch{i}(M) = \epch{i}(\ftil{i}^J M)\}.
\end{equation}
\end{definition}

\begin{lemma} \cite{LV11} \label{jump-lemma}
Let $M$ be a simple $R(\nu)$-module. The following are equivalent:
\begin{enumerate}
\item $\jump{i}(M) = 0$
\item \label{switch-ftil-ftilch} $\ftil{i} M \cong \ftilch{i}M$
\item $\ind M \boxtimes \Lii{i^m}$ is irreducible for all $m \geq 1$
\item $\ind M \boxtimes \Lii{i^m} = \ind \Lii{i^m} \boxtimes M$ for all $m \geq 1$
\item \label{jump-formula} $\wti{i}(M) + \ep{i}(M) + \epch{i}(M) = 0$, where $\wti{i}(M) = -\langle h_i, \nu \rangle$.
\item \label{jump-lemma-6} $\ep{i}(\ftilch{i}M) = \ep{i}(M) + 1$
\item $\epch{i}(\ftil{i}M) = \epch{i}(M) + 1$
\end{enumerate}
\end{lemma}

\begin{proof}
See \cite{LV11}.
\end{proof}

It is shown in \cite{LV11} that 
\begin{equation} \label{ftil-and-jump-ob}
\jump{i}(\ftil{i}M) = \max\{0,\jump{i}(M) -1\}= \jump{i}(\ftilch{i}M).
\end{equation}
It is also shown in \cite{LV11} that 
\begin{equation} \label{eqn-jump}
\jump{i}(M) =  
\wti{i}(M) + \ep{i}(M) + \epch{i}(M).
\end{equation}
Using information from $\jump{i}$ we can also determine when the crystal operators commute with their $\sigma$-symmetric versions.

\begin{example}
Suppose $\ell > 2$. Observe $\jump{1}(\Lii{0}) = 1$ and 
\begin{equation}
\ftilch{1}\ftil{1} \Lii{0} \cong \ind \Lii{1} \boxtimes \T{0,1}
\end{equation}
whose character has support $\{[1,0,1],[0,1,1],[0,1,1]\}$. However 
\begin{equation}
\ftil{1}\ftilch{1} \Lii{0} \cong \ind \Si{1,0} \boxtimes \Lii{1}
\end{equation}
whose character has support $\{[1,0,1],[1,1,0],[1,1,0]\}$.
\end{example}

In the case $\ell = 2$, note $\jump{1}(\T{0,1}) = 1$ and we can similarly calculate $\ftilch{1} \ftil{1} \T{0,1} \not\cong \ftil{1}\ftilch{1} \T{0,1}$ (in fact the former is 8-dimensional while the latter is 4-dimensional).

We shall see below that this phenomenon is special to $\jump{i}(M) = 1$.


\begin{lemma} \label{commuting-functors}
Let $M$ be a simple $R(\nu)$-module. 
\begin{enumerate} 
\item \label{diff-i-j} \cite{LV11}
If $i \neq j$, then
\begin{enumerate} 
\item \label{2-fs} $\ftil{i}\ftilch{j}M \cong \ftilch{j}\ftil{i}M.$ 
\item \label{e-and-f} If $\etilch{j}M \neq \zero$ then
$\ftil{i}\etilch{j}M \cong \etilch{j}\ftil{i}M$. 
\item \label{f-and-e}
If $\etil{j}M \neq \zero$ then $\ftilch{i}\etil{j}M \cong
\etil{j}\ftilch{i}M$.
\item \label{2-es} If further $\etil{i}M \neq \zero$ then,
$\etil{i}\etilch{j}(M) \cong \etilch{j}\etil{i}(M)$. 
\end{enumerate}

\omitt{\item $\jump{i}(\etilch{i}M) = 0$ then $\ftil{i}\etilch{i}(M) \cong \etilch{i}\ftil{i}(M).$}

\item \label{commuting-f-same-i} 
\begin{enumerate}
\item \label{2-fs-same} $\jump{i}(M) \neq 1$ if and only if $\ftilch{i} \ftil{i} M \cong \ftil{i}\ftilch{i} M$. 
\item \label{etilchi-ftil-same} If $\etilch{i}M \neq \zero$, then $\jump{i}(\etilch{i}M) \neq 1$ if and only if $\etilch{i}\ftil{i}M \cong \ftil{i}\etilch{i}M$.
\item \label{etil-ftilch-same} If $\etil{i}M \neq \zero$, then $\jump{i}(\etil{i}M) \neq 1$ if and only if $\etil{i}\ftilch{i}M \cong \ftilch{i}\etil{i}M$.
\end{enumerate}
\end{enumerate}
\end{lemma}

\begin{proof} \mbox{}
\begin{enumerate}
\item Consider the short exact sequence,
\begin{equation}
\zero \rightarrow K \rightarrow \ind M \boxtimes \Lii{i} \rightarrow \ftil{i}M \rightarrow \zero
\end{equation}
and recall $\ftil{i}M$ is the unique composition factor of $\ind M
\boxtimes \Lii{i}$ such that $\ep{i}(\ftil{i}M) = \ep{i}(M) + 1$, and
that for all composition factors $N$ of $K$, $\ep{i}(N) \leq
\ep{i}(M)$.
By the exactness of induction there is a second short exact sequence
\begin{equation} \label{induced-exact-sequence}
\zero \rightarrow \ind \Lii{j} \boxtimes K \rightarrow \ind \Lii{j} \boxtimes M \boxtimes \Lii{i} \rightarrow \ind \Lii{j} \boxtimes \ftil{i}M \rightarrow \zero,
\end{equation}
and since $i \neq j$ the Shuffle Lemma tells us that for all
composition factors $N'$ of $\ind \Lii{j} \boxtimes K$, $\ep{i}(N')
\leq \ep{i}(M)$. 
By the Shuffle Lemma and  Frobenius reciprocity
 \begin{equation}
\ep{i}(\ftilch{j}\ftil{i}M) =
\ep{i}(\ftil{i}\ftilch{j}M) =
\ep{i}(M) + 1.
\end{equation}
Hence there can be no nonzero map
\begin{equation}
\ind \Lii{j} \boxtimes K \rightarrow
\ftil{i} \ftilch{j} M,
\end{equation} 
so that
the submodule $\ind \Lii{j} \boxtimes K$ is contained in the kernel of
$\beta$, as pictured in \eqref{commuting-ftils-diagram}.

\begin{equation}
 \label{commuting-ftils-diagram}
\begin{tikzpicture}
\node at (0,0) {$\ind \Lii{j} \boxtimes M \boxtimes \Lii{i}$};
\node at (5,.7) {$\ind \Lii{j} \boxtimes \ftil{i}M$};
\node at (5,-.7) {$\ind \ftilch{j}M \boxtimes \Lii{i}$};
\node at (9,.7) {$\ftilch{j}\ftil{i}M$};
\node at (9,-.7) {$\ftil{i}\ftilch{j}M$};
\node at (0,-1.5) {};
\node at (5,2.5) {$\alpha$};
\node at (5,-2.5) {$\beta$};
\draw[thick,->>] (2,.2) -- (3.2,.5);
\draw[thick,->>] (2,0) -- (3.2,-.5);
\draw[thick,->>] (7,.73) -- (8,.73);
\draw[thick,->>] (7,-.73) -- (8,-.73);
\draw[thick,->>] (1,.3) to [out = 50, in = 150] (8.8,1.2);
\draw[thick,->>] (1,-.3) to [out = -50, in = -150] (8.8,-1.2);
\end{tikzpicture} 
\end{equation}
Hence $\beta$ induces 
 a nonzero map (necessarily surjective)
\begin{equation}
\ind \Lii{j} \boxtimes \ftil{i}M \twoheadrightarrow \ftil{i}\ftilch{j}M.
\end{equation}
Because $\ind \Lii{j} \boxtimes \ftil{i} M$ has unique simple quotient
$\ftilch{j}\ftil{i}M$, then $\ftilch{j}\ftil{i}M \cong
\ftil{i}\ftilch{j}M$. This proves \ref{2-fs}.

The three isomorphisms in \ref{e-and-f}, \ref{f-and-e}, and \ref{2-es}
all follow from \ref{2-fs}. For example, if $\etilch{j}M$ is nonzero,
then
\begin{equation}
\ftil{i}M \cong \ftil{i}\ftilch{j}\etilch{j}M \cong
\ftilch{j}\ftil{i}\etilch{j}M. 
\end{equation}
Applying $\etilch{j}$ to both sides we get \ref{e-and-f}. \ref{f-and-e} and \ref{2-es} follow similarly.

\item
We prove \ref{2-fs-same}. Let $c = \epch{i}(M), m = \ep{i}(M)$.
\begin{itemize}
	\item Suppose $\jump{i}(M) = 0$. Then also $\jump{i} (\ftil{i}
M) = \jump{i} (\ftilch{i} M) = 0$ by \eqref{ftil-and-jump-ob}. Thus by
Lemma \ref{jump-lemma}
\begin{equation}
\ftilch{i}\ftil{i}M \cong \ftil{i}\ftil{i}M \cong \ftil{i}\ftilch{i}M.
\end{equation}
  
  \item Suppose $\jump{i}(M) = 1$. By Lemma \ref{jump-lemma} and
Proposition \ref{crystal-op-facts}, $\ep{i}(\ftil{i}M) = m+1$ but
$\epch{i}(\ftil{i}M) = c$. While $\ep{i}(\ftilch{i}M) = m$ but
$\epch{i}(\ftilch{i}M) = c+1$. Further by \eqref{ftil-and-jump-ob}
$\jump{i}(\ftil{i}M) = \jump{i} (\ftilch{i} M) = 0$. Hence
$\ep{i}(\ftilch{i}\ftil{i}M) = m+2$, $\epch{i}(\ftilch{i} \ftil{i} M) =
c+1$ whereas $\ep{i}(\ftil{i}\ftilch{i}M) = m+1$, $\epch{i}(\ftil{i}
\ftilch{i}M) = c+2$. Thus the two modules cannot be isomorphic.
  
\item
Suppose $\jump{i}(M) \geq 2$. Then $\jump{i}(\ftil{i}M) =
\jump{i}(\ftilch{i}M) \geq 1$. We calculate
\begin{align}
\ep{i}(\ftil{i}\ftilch{i} M) = m+1 = \ep{i}(\ftilch{i} \ftil{i} M) \\
\epch{i}(\ftil{i} \ftilch{i} M) = c+1 = \epch{i}(\ftilch{i} \ftil{i} M).
\end{align}
We will show there is no nonzero map 
\begin{equation}
\ind \Lii{i} \boxtimes K \rightarrow \ftil{i} \ftilch{i} M 
\end{equation}
for any proper submodule $K \subseteq \ind M \boxtimes \Lii{i}$.
Given we have a surjection 
\begin{equation}
\ind \Lii{i} \boxtimes M \boxtimes \Lii{i} \twoheadrightarrow
\ftil{i} \ftilch{i} M
\end{equation}
this means we must have a nonzero map
\begin{equation}
\ind \Lii{i} \boxtimes \ftil{i}M \rightarrow \ftil{i}\ftilch{i}M,
\end{equation}
which will prove the lemma as 
\begin{equation}
\ftilch{i} \ftil{i} M = \cosoc \ind \Lii{i} \boxtimes \ftil{i}M.
\end{equation}
First note there is no nonzero map
\begin{equation}
\ind \Lii{i} \boxtimes \ftilch{i}M \rightarrow \ftil{i} \ftilch{i}M
\end{equation}
as $\cosoc( \ind \Lii{i} \boxtimes \ftilch{i}M) = (\ftilch{i})^2 M$ and
$\epch{i}((\ftilch{i})^2 M) = c+ 2 \neq c+1 = \epch{i}(\ftil{i}
\ftilch{i}M)$. Let $D$ be any other composition factor of $\ind M
\boxtimes \Lii{i}$ apart from $\ftil{i}M$ or $\ftilch{i}M$ (recall the
latter occur with multiplicity one as composition factors). Then by
Proposition \ref{crystal-op-facts}, $\ep{i}(D) \leq m$,  $\epch{i}(D)
\leq c$. If there were a nonzero map $\ind \Lii{i} \boxtimes D
\rightarrow \ftil{i} \ftilch{i}M$, it would imply $\ftilch{i} D \cong
\ftil{i}\ftilch{i} M$ and so $\epch{i}(\ftilch{i}D) = c+1$ meaning
$\epch{i}(D) = c$. Also $m+1 = \ep{i}(\ftilch{i}D) \leq \ep{i}(D) + 1$
by the Shuffle Lemma, forcing $\ep{i}(D) = m$. By Lemma
\ref{jump-lemma} this forces $0 = \jump{i}(D)$ and $\ftil{i}D \cong
\ftilch{i}D \cong \ftil{i}\ftilch{i} M$ from above, forcing $D \cong
\ftilch{i}M$, which we already ruled out. Hence there must be a nonzero
map
\begin{equation}
\ind \Lii{i} \boxtimes \ftil{i} M \rightarrow \ftil{i} \ftilch{i} M.
\end{equation}
Now that we have established $\ftilch{i} \ftil{i} M \cong
\ftil{i}\ftilch{i} M$ if and only if $\jump{i}(M) \neq 1$,
statements \ref{etilchi-ftil-same} and \ref{etil-ftilch-same}
follow directly from Proposition
\ref{crystal-op-facts}.\ref{e-and-f-undo-eachother}.

\end{itemize}

\end{enumerate}
\end{proof}

\begin{remark} \label{ep-and-ftil-for-i-neq-j}
Because $\etilch{i}$ and $\ftil{j}$ commute for $i \neq j$,
then $\epch{i}(\ftil{j}M) = \epch{i}(M)$. An equivalent statement holds for $\etil{i}$, $\ftilch{j}$, and $\ep{i}$.
When $\jump{i}(M) \neq 0$, $\epch{i}(\ftil{i}M) = \epch{i}(M)$.
\end{remark}


\section{The functor $\pr{}$}
\label{sec-pr}

For $\Lambda = \sum_{i \in I} \lambda_i \Lambda_i \in P^+$ define
$\cycloI{\Lambda}_\nu$ to be the two-sided ideal of $R(\nu)$
generated by the elements $x_1^{\lambda_{i_{1}}}1_{\und{i}}$ for all
$\und{i} \in \seq(\nu)$. When $\nu$ is clear from the context we
write, $\cycloI{\Lambda}_\nu = \cycloI{\Lambda}$. The
{\emph{cyclotomic KLR algebra of weight $\Lambda$}} is then defined
as \begin{equation}
R^\Lambda = \bigoplus_{\nu \in Q^+} R^{\Lambda}(\nu) \quad \text{where} \quad R^\Lambda(\nu) := R(\nu)/\cycloI{\Lambda}_\nu.
\end{equation}
 The algebra $R^\Lambda(\nu)$ is finite dimensional, \cite{BK09a, LV11}.
The category of finite dimensional
 $R^{\Lambda}(\nu)$-modules is denoted $R^{\Lambda}(\nu) \Mod$ and the
category of finite dimensional $R^\Lambda$-modules is denoted
$R^{\Lambda} \Mod$.
The category of finite dimensional $R$-modules on which
$\cycloI{\Lambda}$ vanishes is denoted
$$\rep{\Lambda}.$$
While we can identify $R^{\Lambda} \Mod$ with $\rep{\Lambda},$
we choose to work with $\rep{\Lambda}.$
We
construct a right-exact functor, $\pr{}: R(\nu) \Mod \rightarrow
R(\nu)\Mod$,  via
\begin{equation} \pr{}M := M/\cycloI{\Lambda}M.
\end{equation}
It is customary in the literature to interpret $\pr{}$ as being a
functor from $R(\nu) \Mod$ to $R^{\Lambda}(\nu) \Mod$, but in this
paper it will be more convenient to
consider it as a functor $R(\nu) \Mod \to \rep{\Lambda}$ .
The reader may keep in mind that the image of $\pr{}$ consists of
$R(\nu)$-modules which descend to $R^{\Lambda}(\nu)$-modules.
Observe
that in the opposite direction there is an exact functor $\infl{}:
R^{\Lambda}(\nu) \Mod \rightarrow R(\nu) \Mod$, where $R(\nu)$ acts on
$R^{\Lambda}(\nu)$-module $M$ through the projection map $R(\nu)
\twoheadrightarrow R^{\Lambda}(\nu)$.

\begin{remark} \label{sujrection-onto-pr-remark}
If $M$ is a $R(\nu)$-module and $A$ is a simple
module in $\rep{\Lambda}$ for $\Lambda \in P^+$, then since $\pr{}A
\cong A$, the right exactness of $\pr{}$ implies that any surjection
$M \twoheadrightarrow A$ gives a surjection
$\pr{}M \twoheadrightarrow A$.
Similarly, since there always exists a surjection $M \twoheadrightarrow
\pr{}M$, given a surjection $\pr{}M \twoheadrightarrow A$ we
immediately get a surjection $M \twoheadrightarrow A$. In such
situations there is an equivalence between the two surjections $M
\twoheadrightarrow A$ and $\pr{}M \twoheadrightarrow A$ which we will
henceforth use freely. 
\end{remark}

If $M$ is simple then either $\pr{}M = \zero$ or $\pr{}M = M$. There is
a useful criterion for determining the action of $\pr{}$ on simple
$R(\nu)$-modules given by the following proposition.

\begin{proposition} \label{cyclotomic-char} \cite{LV11}
Let $\Lambda
= \sum_{i \in I} \lambda_i \Lambda_i \in P^+$, $\nu \in Q^+$,
 and let $M$ be a
simple $R(\nu)$-module. Then $\cycloI{\Lambda} M = \zero$ if and
only if $\pr{}M = M$ if and only if $\pr{}M  \neq \zero$ if and only if
$$\epch{i}(M) \leq \lambda_i$$
for all $i \in I$.
When these conditions hold
$M \in \rep{\Lambda}$. Hence we may identify $M$ with $\pr{}M$
(or as an $R^{\Lambda}(\nu)$-module).
\end{proposition}

In this paper we will primarily consider $\Lambda =
\Lambda_i$ in which case $\cycloI{\Lambda_i}_\nu$ is generated by
$x_11_{ii_2 \dots i_m}$ and $1_{ji_2 \dots i_m}, j \neq i$
 ranging over $\und{i}\in \seq(\nu)$.

Notice that Proposition \ref{cyclotomic-char} immediately tells us
that the 1-dimensional modules $\Tii{i}{k} \in \rep{\Lambda_i}$ for
any $k \geq 0$. For $\Lambda = \sum_{i \in I}\lambda_i \Lambda_i \in
P^+$ and $M$ an irreducible $R(\nu)$-module
set
\begin{equation} \label{phcyclo-formula} \phcyc{}{i}(M)
= \lambda_i + \ep{i}(M) + \wti{i}(M). 
\end{equation}
Notice that
when $\Lambda = \Lambda_j$ this gives
\begin{equation}
\label{special-phcyclo-formula}
\phcyc{j}{i}(M) = \delta_{ij} + \ep{i}(M) + \wti{i}(M).
\end{equation}

\begin{remark} \label{when-phcyc-jump-the-same}
By formula \eqref{special-phcyclo-formula} if $M$ is a simple
module in $\rep{\Lambda_j}$ it follows that
\begin{equation}
\phcyc{j}{i}(M) =
	 \begin{cases}
\delta_{ij} & \text{if  }M = \UnitModule, \\
\jump{i}(M) & \text{otherwise.} \\
	\end{cases}
\end{equation}
\end{remark}

\begin{proposition} \label{interp-of-phcyc} Let $M$ be a simple
$R(\nu)$-module with $\pr{}M \neq \zero$. Then
\begin{equation}
\phcyc{}{i}(M) = \max\{n \in \mathbb{Z} \; | \; \pr{}\ftil{i}^n M
\neq \zero \}. 
\end{equation}
\end{proposition}

From property \eqref{ftil-and-jump-ob} of $\jump{i}$ it is clear
that if we apply $\ftil{i}$ sufficiently many times to any module $M
\in R^{\Lambda}(\nu) \Mod$, then eventually we will eventually reach
an $n$ for   which 
\begin{equation} \epch{i}(\ftil{i}^n M) >
\lambda_i \end{equation}
and so $\pr{}\ftil{i}^n M = \zero$.
Proposition \ref{interp-of-phcyc} shows that $\phcyc{}{i}$ measures
this  for simple
modules in $\rep{\Lambda}$.
In fact it is true that
$\pr{}M \neq \zero$ if and only if $\phcyc{}{i}(M) \geq 0$ for all
$i \in I$.
We remark below that the function $\phcyc{}{i}$ is part
of a crystal datum.

\subsubsection{Module-theoretic model of  $B(\Lambda)$}
\label{sec-modulecrystalmodel}

Let $M$ be a simple $R(\nu)$-module.  Set
\begin{equation}
\wt(M) =
- \nu \quad \text{ and } \quad \wti{i}(M) = -\langle h_i, \nu
\rangle. 
\end{equation}
Let $\Irr R$ be the set of isomorphism
classes of simple $R$-modules and $\Irr R^\Lambda$ be the set of
isomorphism classes of simple
modules in $\rep{\Lambda}$.
In \cite{LV11}
it was shown that the tuple  $(\Irr R, \ep{i}, \p_i, \etil{i},
\ftil{i},  \wt)$ defines a crystal isomorphic to $B(\infty)$ and
$(\Irr R^\Lambda, \ep{i}, \phcyc{}{i}, \etil{i}, \ftil{i},  \wt)$
defines a crystal isomorphic to the highest weight crystal
$B(\Lambda)$.

\subsection{Interaction of $\pr{}$ and induction}
\label{sec-pr-ind}

The following is a list of useful facts about
the way that the functor $\pr{}$ interacts with the
functor of induction.

\begin{proposition} \label{pr-facts}
Fix $\Lambda \in P^+$, 
let $\mu, \nu \in Q^+$, $M$ be a simple
$R(\mu)$-module and $N$ a simple $R(\nu)$-module.
\begin{enumerate} [(a)] \item \label{part-1-pr-facts} If
$\pr{}M = \zero$ then $\pr{} \ind M \boxtimes N = \zero$. 
\item \label{vanishing-pr-prop}
If $\pr{}\ind M \boxtimes \Lii{i^c} =
\zero$ and $\epch{i}(N) \ge c$ then $\pr{} \ind M \boxtimes N = \zero$.
\item \label{pr-fact-5} If $c > \phcyc{}{i}(M)$ then $\pr{}\ind M
\boxtimes \Lii{i^c} = \zero$.
\item \label{f-and-pr} Let $\p = \p_i^\Lambda (M)$, then $\pr{}\ind M \boxtimes \Lii{i^\p} \cong \ftil{i}^\p M$.
\item \label{pr-fact-4} If $\pr{}C = M$ then $\pr{}\ind C \boxtimes N \cong \pr{}\ind M \boxtimes N$
\end{enumerate}
\end{proposition}

\begin{proof} We let
\begin{equation}
\Lambda = \sum_{i \in I} \lambda_i \Lambda_i.
\end{equation}
\begin{enumerate} [(a)]

\item If $\pr{}M = \zero$, then by Proposition \ref{cyclotomic-char} there is some $i \in I$ such that $\epch{i}(M) > \lambda_i$. Suppose that $\pr{}\ind M \boxtimes N \neq \zero$. Then it has some
simple quotient $Q$, and there are surjections
\begin{equation} \label{double-surjection-to-simple}
\ind M \boxtimes N \twoheadrightarrow \pr{}\ind M \boxtimes N \twoheadrightarrow Q.
\end{equation}
Frobenius reciprocity
and Proposition \ref{cyclotomic-char} imply that $\pr{}Q = Q$. By Frobenius reciprocity $\res_{\mu,\nu}^{\mu + \nu}Q$ has $M \boxtimes N$ as a $(R(\mu) \otimes R(\nu))$-submodule. But Remark \ref{rem-char-epsilon} then implies $\epch{i}(Q) > \lambda_i$ so that $\pr{}Q = \zero$, a contradiction.
\\
\item If $\epch{i}(N) \ge c$ then there is a surjection,
\begin{equation}
\ind \Lii{i^c} \boxtimes (\etilch{i})^c N \twoheadrightarrow N
\end{equation}
and by the exactness of induction a surjection
\begin{equation}
\ind M \boxtimes \Lii{i^c} \boxtimes (\etilch{i})^c N
\twoheadrightarrow \ind M \boxtimes N.
\end{equation}
If $\pr{}\ind M \boxtimes \Lii{i^c} = \zero$, then by part
\eqref{part-1-pr-facts} above and the right exactness of $\pr{}$,
$\pr{} \ind M \boxtimes N = \zero$.  \\
\item
This follows from Proposition \ref{interp-of-phcyc} and 
the fact that
the induced module has unique simple quotient $\ftil{i}^c M$;
or see \cite{LV11}.
\item Consider the exact sequence,
\begin{equation}
\zero \rightarrow K \rightarrow \ind M \boxtimes L(i^{\p}) \rightarrow \ftil{i}^{\p}M \rightarrow \zero.
\end{equation}
$\ftil{i}^\p M$ is the unique composition factor of $\ind M \boxtimes
\Lii{i^\p}$ such that $\ep{i}(\ftil{i}^\p M) = \p + \ep{i}(M)$, so
$\ep{i}(D) < \p + \ep{i}(M)$ for all composition factors $D$ of $K$
by Proposition \ref{crystal-op-facts}.
All composition factors $D$ of $K$ have the same weight as $\ftil{i}^\p
M$.
By \eqref{phcyclo-formula} and Proposition \ref{crystal-op-facts},
$\phcyc{}{i}(D) = \lambda_i + \ep{i}(D) + \wti{i}(D) 
	< \lambda_i + \ep{i}(\ftil{i}^\p M) + \wti{i}(\ftil{i}^\p M)
	= \phcyc{}{i}(\ftil{i}^\p M) = 0$.
In particular this shows $\pr{}K=\zero$ so by the right exactness
of $\pr{}$ we get \eqref{f-and-pr}.

\item Consider the diagram in \eqref{pr-diagram-1},
\begin{equation} \label{pr-diagram-1}
\begin{tikzpicture}
\node at (0,0) {$\zero$};
\node at (0,1.5) {$\ind M \boxtimes N$};
\node at (0,3) {$\ind C \boxtimes N$};
\node at (0,4.5) {$\ind \cycloI{\Lambda}C \boxtimes N$};
\node at (0,6) {${\zero}$};
\node at (-6.5,3) {${\zero}$};
\node at (-3.5,3) {$\cycloI{\Lambda}(\ind C \boxtimes N)$};
\node at (3.5,3) {$\pr{}(\ind C \boxtimes N)$};
\node at (6.5,3) {$\zero$};
\node at (-.45,3.75) {$\alpha$};
\node at (1.4,3.3) {$\beta$};
\node at (2.2,4.3) {$\beta \circ \alpha$};
\node at (1.8,1.4) {$g$};
\node at (-.3,2.3) {$\gamma$};
\draw[thick,<-] (0,.3) -- (0,1.1);
\draw[thick,<-] (0,3.3) -- (0,4.1);
\draw[thick,<-] (0,1.8) -- (0,2.6);
\draw[thick,<-] (0,4.8) -- (0,5.6);
\draw[thick,->] (-6.1,3) -- (-4.8,3);
\draw[thick,->] (-2.1,3) -- (-1.1,3);
\draw[thick,->] (1.1,3) -- (2.1,3);
\draw[thick,->] (4.8,3) -- (6.1,3);
\draw[thick,->] (1.2,4.3) -- (2.5,3.4);
\draw[thick,dotted,->] (1.1,1.5) -- (2.5,2.5);
\end{tikzpicture}
\end{equation}
where the horizontal and vertical sequences are exact. Recall that
$\cycloI{\Lambda}_\mu$ in $R(\mu)$ is generated by the set
$\{x_1^{\lambda_{i_{1}}}1_{\und{i}}\}_{\und{i} \in \seq(\mu)}$ where
$\und{i} = i_1i_2 \dots i_m$ and $|\mu| = m$. Under the embedding
\begin{equation}
R(\mu) \hookrightarrow R(\mu) \otimes R(\nu) \hookrightarrow R(\mu + \nu),
\end{equation}
this set maps to the set 
\begin{equation}
\Big\{\sum_{\und{j} \in \seq(\nu)}x_1^{\lambda_{i_{1}}}1_{\und{ij}}\Big\}_{\und{i} \in \seq(\mu)}
\end{equation}
in $R(\mu + \nu)$. This set is contained in the ideal generated by $\{x_1^{\lambda_{i_{1}}}1_{\und{k}}\}_{\und{k} \in \seq(\mu + \nu)}$ which generates $\cycloI{\Lambda}_{\mu + \nu}$. It follows that
\begin{equation}
R(\mu + \nu)\cycloI{\Lambda}_\mu \subseteq \cycloI{\Lambda}_{\mu+\nu},
\end{equation}
and hence
\begin{equation}
\ind \cycloI{\Lambda}_\mu C \boxtimes N \subseteq \cycloI{\Lambda}_{\mu + \nu}(\ind C \boxtimes N).
\end{equation}
This tells us that the composition $\beta \circ \alpha$ from the
diagram in \eqref{pr-diagram-1} is zero, so there exists a
surjective homomorphism $g: \ind M \boxtimes N \rightarrow \pr{}\ind C
\boxtimes N$. Applying $\pr{}$ to the diagram 
\eqref{pr-diagram-1}, and denoting the resulting
maps from  $\gamma$, $\beta$, and $g$ as
$\tilde{\gamma}$, $\tilde{\beta}$, and $\tilde{g}$ respectively,
right exactness yields 
$\tilde{\gamma}$, $\tilde{\beta}$, and $\tilde{g}$ are surjections
as shown in \eqref{pr-diagram-2}. It follows from considerations
of dimension and that $\pr{}C=M$
 that $\tilde{g}$ must be an isomorphism.
\begin{equation}
\begin{tikzpicture} \label{pr-diagram-2}
\node at (0,1.5) {$\pr{}(\ind M \boxtimes N)$};
\node at (0,3) {$\pr{}(\ind C \boxtimes N)$};
\node at (3.5,3) {$\pr{}(\ind C \boxtimes N)$};
\node at (6.5,3) {$\zero$};
\node at (1.7,3.5) {$\tilde{\beta}$};
\node at (2.2,1.6) {$\tilde{g}$};
\node at (-.3,2.3) {$\tilde{\gamma}$};
\node at (0,0) {$\zero$};
\draw[thick,<-] (0,.3) -- (0,1.1);
\draw[thick,<-] (0,2) -- (0,2.6);
\draw[thick,->] (1.4,1.7) -- (2.5,2.5);
\draw[thick,->] (5,3) -- (6.1,3);
\draw[thick,->] (1.5,3) -- (2,3);
\end{tikzpicture}
\end{equation}
\end{enumerate}
\end{proof}

\subsection{Applying Proposition \ref{pr-facts} to $\Tii{i}{k}$}
We will frequently need to compute $\jump{j}$ for the 1-dimensional
``trivial" $R(\gammaplus{i}{k})$-module $\Tii{i}{k}$. When $k = 0$, we
compute
for the unit module that
 $\jump{j}(\UnitModule) = 0$ but
$\phcyc{i}{j}(\UnitModule) = \delta_{ij}$. When $k \geq 1$,
\begin{equation} \label{trig-weight}
\wti{j}(\Tii{i}{k}) = -\langle h_j, \gammaplus{k}{i} \rangle = \delta_{j,i-1} - \delta_{j,i} + \delta_{j,i+k} - \delta_{j,i+k-1}.
\end{equation}
Note that here as elsewhere, the indices $p,q$ in $\delta_{p,q}$ should be taken modulo $\ell$. Then,
\begin{equation} \label{jump-formula-triv}
\jump{j}(\T{i, i+1, \dots , i+k-1}) = \delta_{j,i-1} + \delta_{j,i+k}.
\end{equation}
Similarly,
\begin{equation} \label{phcyc-formula-triv}
\phcyc{i}{j}(\T{i,i+1, \dots , i+k-1}) = \delta_{j,i-1} + \delta_{j,i+k}.
\end{equation}

Here we record some useful facts concerning the way that the modules
$\Tii{0}{k}$ interact with the functors induction and $\pr{0}$ .
Notice that all these facts hold for $\Tii{i}{k}$ and
$\pr{i}$ after making obvious modifications.

\begin{proposition} \label{triv-pr-Lemma} \mbox{}
Fix $k \in \N$, $k> 0$.
\begin{enumerate} [1.]
	
\item \label{triv-pr-Lemma-1}
If $j \not\equiv -1, k $ then
\begin{math} \pr{0} \ind \T{0, 1, \dots, k-1} \boxtimes \Lii{j}
= \zero.  \end{math}

\item \label{triv-pr-Lemma-2}
 If $k \not\equiv -1 $ then
\begin{math} \pr{0} \ind \T{0,1,\dots,k-1} \boxtimes \Lii{k} \cong
\T{0 , 1 , \dots , k}.  \end{math}
	
\item \label{double-f}
If $k \not\equiv -1 $ then
\begin{math} 
\pr{0}\ind \T{0, 1, \dots, k-1} \boxtimes \Lii{k} \boxtimes \Lii{k} = \zero.  
\end{math}
	
\item \label{triv-pr-Lemma-4}
\begin{math}
\pr{0} \ind \Tii{0}{k} \boxtimes \Lii{-1} \cong
\Bigl(\ind \Tii{0}{k} \boxtimes \Lii{-1}\Bigr)
	\Big/\T{-1 , 0 , \dots , k-1}
\end{math} and further
	\begin{enumerate}

\item \label{k-neq-(-1)-part1}
If $k \not\equiv -1 $ then
$\pr{0} \ind \Tii{0}{k} \boxtimes \Lii{-1}$ is irreducible and
\begin{equation*}
\pr{0} \ind \Tii{0}{k} \boxtimes \Lii{-1} \boxtimes \Lii{-1} = \zero.
\end{equation*}

\item \label{k-eq-(-1)-part1}
If $k\equiv -1$, then
\begin{gather*} \label{double-f-1-a}
\pr{0}\ind \Tii{0}{k} \boxtimes \Lii{-1} \boxtimes \Lii{-1} \cong
\ftil{-1}^2 \Tii{0}{k} 
\\
\pr{0} \ind \Tii{0}{k} \boxtimes \Lii{-1} \boxtimes \Lii{-1} \boxtimes
\Lii{-1} = \zero.  
\end{gather*}
Further, if $k>1$ then 
$\pr{0} \ind \Tii{0}{k} \boxtimes \Lii{-1}$
has two composition factors. 
But if $k=1$ (so $\ell = 2$) then 
$\pr{0} \ind \Tii{0}{k} \boxtimes \Lii{-1} = \Ti(0,1)$
is irreducible.

\end{enumerate}
\end{enumerate}
\end{proposition}

\begin{proof} \mbox{}
Note the hypothesis $k > 0$ implies $\Tii{0}{k} \neq \UnitModule$,
i.e.\ $\nu \neq 0$.

Below we write $=$ for $\equiv \bmod \ell$ or equality in $I$.
\begin{enumerate}[1.] 

\item When $j \neq -1, k$, formula \eqref{phcyc-formula-triv} gives
\begin{equation}
\phcyc{0}{j}(\T{0, 1, \dots, k-1}) = 0.
\end{equation}
Hence by Proposition \ref{pr-facts}.\ref{pr-fact-5} 
\begin{equation}
\pr{0} \ind \T{0, 1, \dots, k-1} \boxtimes \Lii{j} = \zero.
\end{equation}

\item When $k \neq -1$, formula \eqref{phcyc-formula-triv} gives 
\begin{equation}
\phcyc{0}{k}(\T{0, 1, \dots, k-1}) = 1.
\end{equation}
Hence by Proposition \ref{pr-facts}.\ref{f-and-pr}
\begin{align}
\pr{0}\ind \T{0,1, \dots, k-1} \boxtimes \Lii{k} \cong \ftil{k}\T{0, 1, \dots, k-1} \\
\cong \T{0, 1, \dots, k-1,k},
\end{align}
where the second isomorphism holds by Frobenius reciprocity and the irreducibility of $\ftil{k} \Tii{0}{k}$.

\item As noted above
\begin{equation}
\phcyc{0}{k}(\T{0, \dots ,k-1}) = 1.
\end{equation}
Proposition \ref{pr-facts}.\ref{pr-fact-5} then implies,
\begin{equation}
\pr{0}\ind \T{0, \dots ,k-1} \boxtimes \Lii{k} \boxtimes \Lii{k} = \zero.
\end{equation}

\item Let $v \otimes u$ span the 1-dimensional module $\Tii{0}{k} \boxtimes \Lii{-1}$. Then as in \eqref{ind-basis}, \\$M := \ind \Tii{0}{k} \,\boxtimes \, \Lii{-1}$ has basis 
\begin{equation}
\{\; 1_{\und{i}} \otimes (v \otimes u), \; \psi_{k}1_{\und{i}} \otimes (v \otimes u), \; \ldots, \; \psi_1 \cdots \psi_{k-1} \psi_{k} 1_{\und{i}} \otimes (v \otimes u) \; \}
\end{equation}
where $\und{i} = (0,1,\dots,k-1,-1)$.

Recall that $\cycloI{\Lambda_0}$ is generated by $x_1 1_{\und{j}}$
where $j_1 = 0$ and by $1_{\und{p}}$ where $p_1 \neq 0$.

In particular, for $\und{p} = (-1,0,\dots,k-1)$ we see 
\begin{equation}
1_{\und{p}}\Big(\psi_1 \dots \psi_k 1_{\und{i}} \otimes (v \otimes u)\Big) = \psi_1 \dots \psi_k 1_{\und{i}} \otimes (v \otimes u) \in \cycloI{\Lambda_0} M,
\end{equation}
but $1_{\und{p}} (\psi_r \dots \psi_{k-1} \psi_{k} 1_{\und{i}} \otimes (v \otimes u)) = 0$ for $r >1$. Also note $1_{\und{j}} 1_{\und{p}} = 0$.
 We further calculate 
\begin{equation}
x_1 1_{\und{j}} \Big(\psi_r \dots \psi_r 1_{\und{i}} \otimes (v \otimes u)\Big) = 1_{\und{j}} \psi_r \dots \psi_k 1_{\und{i}} \otimes (x_1v \otimes u) = 0
\end{equation}
whenever $r > 1$. Hence $\cycloI{\Lambda_0} M$ is spanned by $\psi_1 \dots \psi_k 1_{\und{i}} \otimes (v \otimes u)$ and so $\pr{0}M = M \big/ \T{-1,0,\dots, k-1}$ as stated.
\begin{enumerate}[4(a)] 

\item Suppose $k \neq -1$. As $\phcyc{0}{-1}(\Tii{0}{k}) = 1$ by
\eqref{phcyc-formula-triv}, Proposition
\ref{pr-facts}.\ref{f-and-pr} tells us that $\pr{0}M = \ftil{-1}
\Tii{0}{k}$ is irreducible. In particular it has dimension $k$.
Using the Shuffle Lemma (along with the calculation of
$\cycloI{\Lambda_0}M$ above), one could 
easily compute its character.
By Proposition \ref{pr-facts}.\ref{pr-fact-5} 
we see
$\pr{0} \ind \Tii{0}{k} \boxtimes \Lii{-1} \boxtimes \Lii{-1} =
\zero.$

\item Next suppose
$ k = -1$ (i.e. $k \equiv -1  \bmod \ell$)
and
$k > 1$ (when considering $k \in \N$).
Then \eqref{phcyc-formula-triv} yields $\phcyc{0}{-1}(\Tii{0}{k}) = 2$
so Proposition \ref{pr-facts}.\ref{f-and-pr} immediately gives
\eqref{double-f-1-a}. 
Further, it is easy to
see $\ftil{-1}\Tii{0}{k} = \ftil{k} \Tii{0}{k} = \Tii{0}{k+1}$ which is
1-dimensional so from above $\pr{0}M$ has at least 2 composition
factors when $k > 1$. Next, the Shuffle Lemma and Serre relations
\eqref{Serre-relations} tell us the $(k-1)$-dimensional subquotient,
corresponding to the span of
\begin{equation}
\{ \psi_k \otimes (v \otimes u), \ldots,  \psi_2 \cdots \psi_k
\otimes (v \otimes u) \}, \end{equation}
is irreducible. 

The remaining case is
$k \equiv -1 \bmod \ell$,
$k =1$,
forcing $\ell = 2$,
and also $\Tii{0}{k} = \Lii{0}$.
As above $\phcyc{0}{-1}(\Tii{0}{1}) = 2$, yielding
\eqref{double-f-1-a}. 
The only difference is that $\pr{0}M$ has only 1 composition factor (namely $\Tii{0}{k+1}$) as
 it is only 1-dimensional.

\end{enumerate}
\end{enumerate}
\end{proof}

%
\section{Main theorems}
\label{sec-main}
%

As remarked in Section \ref{sec-modulecrystalmodel},
the graph with nodes corresponding to isomorphism classes
$[M]$ for $M$ a simple
$R^{\Lambda_i}(\nu)$-module
and
arrows $[\etil{j} M] \xrightarrow{j} [M]$
\omitt{, and additional crystal data $\ep{j}, \phcyc{i}{j},
\wt[M] = -\nu$,}
is the crystal graph $B(\Lambda_i)$.
We can also
use  $\ell$-restricted partitions $\lambda$
\omitt{, i.e.  $\lambda_k - \lambda_{k-1} < \ell$.}
to label the  nodes of $B(\Lambda_i)$ as $[M^\lambda]$.
The main theorems show
for the isomorphism
$\crystalmap:B(\Lambda_{i})
\xrightarrow{\simeq}
 \Bp \otimes B(\Lambda_{i-1})$
that
\begin{tikzpicture}[baseline=-2pt] 
	\node at (0,0) {\;\;\scriptsize{$k\!-\!1\!+\!i$}\;\;}; 
        \draw (0,0) ellipse (.48cm and .26cm);
\end{tikzpicture}$ \; \otimes\; \mu = \crystalmap(\lambda)$
corresponds to
$$\ind \Tii{i}{k} \boxtimes [M^\mu] \twoheadrightarrow
 [M^\lambda]$$
for $k=r(M^\lambda)$ (defined below),
and that the crystal operators commute
with this surjection in the appropriate manner.

Another way to view the theorems is that they give a module-theoretic
construction of $\crystalmap$ and justify it is an isomorphism
of crystals. 

\begin{theorem} \label{exist-theorem}
Let $A$ be a simple
$R (\nu)$-module in $\rep{\Lambda_i}$
with $|\nu| \geq 1$. 
\begin{enumerate}

\item \label{exist-theorem-1}
There exists $k \in \N, k {\geq 1}$
such that $\etilch{i+k-1} \dots \etilch{i+1} \etilch{i} A$ is a simple
$R(\nu - \gammaplus{i}{k})$-module
in $\rep{\Lambda_{i-1}}$.

\item Let 
$$r(A) = k$$
 be the minimal $k$ such that statement \eqref{exist-theorem-1}
holds and let
$$\Rcal{A} = \etilch{i+k-1} \dots \etilch{i+1}
\etilch{i} A.$$
Then there exists a surjection \begin{equation}
\label{exist-theorem-2-surj}
\pr{i}\ind \T{i,i+1, \dots, i+k-1} \boxtimes \Rcal{A} \twoheadrightarrow A.
\end{equation}
\end{enumerate}
\end{theorem}

\begin{proof}
For ease of exposition, we set $i = 0$ in the proof.
For $t \in \N$ set $\Rcalj{0}{A} = A$ and
let
\begin{equation} \label{def-Rt}
\Rcalj{t}{A} = \etilch{t-1} \dots \etilch{1} \etilch{0} A.
\end{equation}
We show by induction on $t \leq r(A)$ that $\Rcalj{t}{A} \in \rep{\Lambda_t + \Lambda_{-1}}$ and there exists a surjection 
\begin{equation} \label{inductive-hypothesis}
\ind \T{0,1, \dots, t-1} \boxtimes \Rcalj{t}{A} \twoheadrightarrow A.
\end{equation}

In the base case $t = 0$, $\Rcalj{0}{A} = A$. If $|\nu| = 1$, then $A = \Lii{0}$ and $\etilch{0}\Lii{0} \cong \UnitModule \in \rep{\Lambda_{-1}}$, so $r(A) = 1$. The existence of the surjection in this case is vacuous. Assume that $|\nu| > 1$. Then $\Rcalj{1}{A} = \etilch{0}A \neq \zero, \UnitModule$. By Proposition \ref{crystal-op-facts}.\ref{ftil-and-ep} there is a surjection
\begin{equation} 
\ind \Lii{0} \boxtimes \Rcalj{1}{A} \twoheadrightarrow A.
\end{equation}

It follows directly
from Proposition \ref{triv-pr-Lemma} and Proposition
\ref{pr-facts}.\ref{vanishing-pr-prop} that
if $\pr{0}\ind \Tii{0}{t} \boxtimes D \twoheadrightarrow A$ and $t \geq
1$ then $D \in \rep{\Lambda_t + \Lambda_{-1}}$.

In more detail, Proposition \ref{cyclotomic-char} implies $D \in \rep{\Lambda}$ where $\Lambda = \sum \epch{i}(D) \Lambda_i \in P^+$. Proposition \ref{pr-facts}.\ref{vanishing-pr-prop} tells us $\pr{0} \ind \Tii{0}{t} \boxtimes D \neq 0$ implies $\pr{0} \ind \Tii{0}{t} \boxtimes \Lii{i^{\epch{i}(D)}} \neq 0$. Thus for $i \neq -1,t$ we have $\epch{i}(D) = 0$ by Proposition \ref{triv-pr-Lemma}.\ref{triv-pr-Lemma-1}.

If $t \neq -1$, Proposition \ref{triv-pr-Lemma}.\ref{k-neq-(-1)-part1} implies $\epch{-1}(D) \leq 1$ and Proposition \ref{triv-pr-Lemma}.\ref{double-f} implies $\epch{t}(D) \leq 1$ so $D \in \rep{\Lambda_t + \Lambda_{-1}}$. If $t = -1$ then Proposition \ref{triv-pr-Lemma}.\ref{k-eq-(-1)-part1} implies $\epch{-1}(D) \leq 2$ so $D \in \rep{2\Lambda_{-1}} = \rep{\Lambda_t + \Lambda_{-1}}$.

Since $\pr{0}A = A$, observe any surjection $M \twoheadrightarrow A$ factors through $M \twoheadrightarrow \pr{0}M \twoheadrightarrow A$. Assume our inductive hypothesis \eqref{inductive-hypothesis} holds. Then from above, $\Rcalj{t}{A} \in \rep{\Lambda_t + \Lambda_{-1}}$. If in fact $\Rcalj{t}{A} \in \rep{\Lambda_{-1}} $then we are done (and $t \geq r(A)$). If not, then $\Rcalj{t+1}{A} = \etilch{t} \Rcalj{t}{A} \neq 0$. 

Transitivity and exactness of induction give us a surjection
\begin{equation} \label{surj-for-trivial}
\ind \T{0, \dots, t-1} \boxtimes \Lii{t} \boxtimes \Rcalj{t+1}{A} \twoheadrightarrow A.
\end{equation}

In the first case, suppose $t \neq -1$. Then Proposition \ref{pr-facts}.\ref{pr-fact-4} and Proposition \ref{triv-pr-Lemma}.\ref{triv-pr-Lemma-2} imply 
\begin{equation} 
\pr{0} \ind \Tii{0}{t} \boxtimes \Lii{t} \boxtimes \Rcalj{t+1}{A} \cong \pr{0} \ind \Tii{0}{t+1} \boxtimes \Rcalj{t+1}{A}
\end{equation}
and we get
\begin{equation}
\pr{0} \ind \T{0, \dots, t-1,t} \boxtimes \Rcalj{t+1}{A} \twoheadrightarrow A.
\end{equation}
In the case $t = -1$ then by the inductive hypothesis $\Rcalj{t}{A} \in
\rep{2\Lambda_{-1}}$ and $\Rcalj{t}{A} \notin \rep{\Lambda_{-1}}$ as we
are assuming $t < r(A)$. Thus $\epch{-1}(\etilch{-1} \Rcalj{t}{A}) =
\epch{-1}(\Rcalj{t+1}{A}) = 1$. If $K$ is any composition factor of
$\ind \Tii{0}{t} \boxtimes \Lii{-1}$ other than $\Tii{0}{t+1}$ then
$\phcyc{0}{-1}(K) \leq 0$ by \eqref{phcyclo-formula} so $\pr{0} \ind K
\boxtimes \Lii{-1} = \zero$ which implies $\pr{0} \ind K \boxtimes
\Rcalj{t+1}{A} = \zero$. So \eqref{surj-for-trivial} must factor through
\begin{equation}
\ind \T{0, \dots, -2,-1} \boxtimes \Rcalj{t+1}{A} \twoheadrightarrow A.
\end{equation}
This completes the induction.

 We take $r(A)$ to be the
smallest $k$ such that $\Rcalj{k}{A} \in \rep{\Lambda_{-1}}$. Note that
the process above must terminate as $r(A) \leq |\nu|$. In fact, in the
case $r(A) = |\nu|$ we must have $A = \Tii{0}{|\nu|}$ and
$\Rcalj{|\nu|}{A} = \Rcal{A} = \UnitModule \in \rep{\Lambda_{-1}}$. 
\end{proof}


By considering sign in place of trivial modules,
a very similar proof yields the following theorem.

\begin{theorem}
\label{thm-sign-main}
Let $A$ be a simple
$R (\nu)$-module in $\rep{\Lambda_i}$
with $|\nu| \geq 1$.
\begin{enumerate}
\item \label{sign-existence-thm}
There exists $k \in \N, k {\geq 1}$
such that $\zero \neq \etilch{i-k+1} \dots \etilch{i-1} \etilch{i} A$
is a simple
$R(\nu - \gammaminus{i}{k})$-module in $\rep{\Lambda_{i+1}}$.
 \item Let $c(A) = k$
be the minimal $k$ such that \eqref{sign-existence-thm} holds and
let $\Ccal{A} = \etilch{i-k+1} \dots \etilch{i-1} \etilch{i} A$. Then
there exists a surjection
\begin{equation}
\pr{i} \ind \Si{i,i-1, \dots, i-k+1} \boxtimes \Ccal{A} \twoheadrightarrow A.
\end{equation}
\end{enumerate}
\end{theorem}
\begin{conjecture}
With hypotheses as above,
\begin{gather*}
A = \cosoc \pr{i} \ind \Tii{i}{r(A)} \boxtimes \Rcal{A},
\\
A = \cosoc \pr{i} \ind \Sii{i}{c(A)} \boxtimes \Ccal{A}.
\end{gather*}
\end{conjecture}

\subsection{Relation to Specht modules}
A
Specht module for $\Sy{n}$ is constructed as a
submodule of the induction of a trivial module from
a Young subgroup $\Sy{\lambda}$ (this is one of our
$\Sy{P}$ as in  \eqref{eq-parabolic}).
Specht modules can also be constructed for the
Hecke algebra of type $A$ as in \cite{DJ}.
They are equipped with an integral form that allows
one to specialize the Specht modules over $\F_\ell$ 
in the former case, to an $\ell$-th root of unity in the latter.

When $\lambda$ is $\ell$-regular, the specialization
of the Specht module $\bar S^\lambda$  has unique
simple quotient $D^\lambda$.  In other words,
$D^\lambda$  is a subquotient of a module induced from a 1-dimensional
module.  Further
$\{D^\lambda \mid 
\lambda \vdash n, \lambda \text{ is $\ell$-regular}\}$
is a complete set of simple modules of $\F_\ell \Sy{n}$
or the finite
Hecke algebra at an $\ell$-th root of unity. 
The crystal structure on these simples by taking
socle of restriction  agrees with the
 model of $B(\Lambda_0)$ taking nodes to be $\ell$-regular 
partitions \cite{Klesh}.  
This is the model compatible with tensoring by
$\Bopp$. 
(See Section \ref{sec-crystal1}.)

Repeating the construction of Theorem \ref{thm-sign-main}
yields $D^\lambda$ as the quotient of a module induced
from a (possibly conjugate) 
parabolic subalgebra of shape $\lambda^T$, where
the module being induced is a (parabolic) sign module. 
In other words, the restriction of $D^\lambda$ to 
that  parabolic subalgebra 
contains a $\boxtimes$ of
sign modules $\Sii{i}{k}$.
On the other hand, in the construction of Specht modules for 
the finite Hecke algebra of type $A$ given in \cite{DJ},
the Specht module contains a special vector that is anti-symmetrized
according to a parabolic subalgebra of shape $\lambda^T$.
In other words, the same induced module that has $D^\lambda$
as a quotient also has a nonzero map to $S^\lambda$. 

In fact $\Q \otimes_\Z S^\lambda$ can be characterized
as the unique irreducible $\Q \Sy{n}$-module such that
$\Res_{\Sy{\lambda}}$ contains a trivial module and
$\Res_{\Sy{\lambda^T}}$ contains a sign module.

When $\ell$ is a root of unity (or we work over $\F_\ell$)
the difficulty is in specializing quotients of (induced)
modules.  The existence of a map from and induced sign
module to $D^\lambda$ is not a surprise, but the
result on how the crystal operators act is nontrivial.

\subsection{The action of crystal operators $\ftil{j}$ and $\etil{j}$}
\label{sec-ef}

Next we study the action of the crystal operators $\etil{j}$ and $\ftil{j}$ to show \eqref{exist-theorem-2-surj} categorifies our crystal isomorphism $\crystalmap$. 
We refer the reader back to Section \ref{sec-crystal1}.

Compare the theorems below with
\eqref{eq_ei_tensor} 
and
\eqref{eq_fi_tensor}.
As in \cite{LV11} simple modules correspond to nodes in
$B(\Lambda_i)$.  Each node of the perfect crystal
$\Bp$ (respectively $\Bopp$)
corresponds to a family of trivial (respectively sign) modules
$\Tii{i}{k + t\ell}, t \in \N$.
(However this does not give a categorification of $\Bp$ itself.)
It is in this manner that the main theorems of this paper
give a categorification of the crystal isomorphism $\crystalmap$
(resp. $\crystalmapopp$).

\begin{theorem} \label{thm-action-etil}
Let $A \in \rep{\Lambda_i}$ be simple.
Let $j \in I$ be such that $\etil{j}A \neq 0$, and let $k = r(A)$. Then there exists a surjection 
\begin{equation}
  \begin{array}{l l}
    \ind \Big(\etil{j} \Tii{i}{k} \boxtimes \Rcal{A}\Big) \twoheadrightarrow \etil{j}A & \quad \text{if   } \;\;\;\; \ep{j}(\Tii{i}{k} )> \phcyc{i-1}{j}(\Rcal{A})
\\
    \ind \Big(\Tii{i}{k} \boxtimes \etil{j}\Rcal{A}\Big) \twoheadrightarrow \etil{j}A & \quad \text{if  }\;\;\;\; \ep{j}(\Tii{i}{k}) \leq \phcyc{i-1}{j}(\Rcal{A}).
  \end{array} 
\end{equation}
\end{theorem}

\begin{theorem} \label{thm-action-ftil}
Let $A \in \rep{\Lambda_i}$ be simple.
Let $j \in I$ be such that $\pr{i} \ftil{j} A \neq 0$, and let $k=
r(A)$.  Then there exists a surjection
\begin{equation}
  \begin{array}{l l}
    \ind \Big(\ftil{j} \Tii{i}{k} \boxtimes \Rcal{A} \Big) \twoheadrightarrow \ftil{j}A & \quad \text{if} \;\;\;\; \ep{j}(\Tii{i}{k}) \geq \phcyc{i-1}{j}(\Rcal{A})
\\
    \ind \Big(\Tii{i}{k} \boxtimes \ftil{j}\Rcal{A}\Big) \twoheadrightarrow \ftil{j}A & \quad \text{if} \;\;\;\; \ep{j}(\Tii{i}{k}) < \phcyc{i-1}{j}(\Rcal{A}).
  \end{array} 
\end{equation}
\end{theorem}

Theorem \ref{thm-action-etil} follows directly from Theorem \ref{thm-action-ftil}, therefore will only prove the latter.
Similar theorems hold using sign modules and
$\Ccal{A}$.

 Before doing this, we need to establish several lemmas.  


\begin{proposition} \label{propn-Vaz} \cite{V07}
Let $m = |\nu|$. Let $M$ be a simple $R(\nu)$-module. If $M \in
\rep{\Lambda_i}$ and $\zero \neq \etilch{i}(M) \in \rep{\Lambda_{i+1}}$ then
$M = \Tii{i}{m}$. 
\end{proposition}

\begin{proof}
This can be directly adapted from Theorem 3.7 of \cite{V07},
replacing $\etil{i}$ with $\etilch{i}$ and noting $\epch{j}(M) =
\delta_{i,j}$, $\epch{j}(\etilch{i}M) = \delta_{i+1,j}$ (assuming 
$m \ge 2$). It was
proved in the context of $B(\Lambda_i)$, hence holds for
$\rep{\Lambda_i}$ by \cite{LV11}.  \end{proof}

\begin{proposition} \label{propn-level-two}
Let $A$, $\Rcalj{t}{A}$ be as in 
\eqref{def-Rt}
 and $m = |\nu|$. If there exists $1 \leq t < r(A)$ with $\Rcalj{t}{A} \in \rep{\Lambda_{i+t}}$, then in fact $\Rcalj{j}{A} \in \rep{\Lambda_{i+j}}$ for all $1 \leq j \leq r(A)$, $r(A) = \min\{\ell-1,m\}$, and $A = \Tii{i}{m}$.
\end{proposition}

\begin{proof}
As usual, we set $i=0$ for ease of exposition. We have already shown $\Rcalj{t}{A} \in \rep{\Lambda_t + \Lambda_{-1}}$. Given $A = \Rcalj{0}{A} \in \rep{\Lambda_0}$, suppose $\ech{0}A = \etilch{0}A = \Rcalj{1}{A} \in \rep{\Lambda_1}$ then by Proposition \ref{propn-Vaz}, $A = \Tii{0}{m}$ and we are done. Further in the case $\ell = 2$ this means $r(A) = 1$. If $\ell > 2$ then $r(A) \leq \ell-1$. Assume otherwise. 

{\emph{Case 1}}: $\ell \neq 2$. Then $\epch{-1}(\Rcalj{1}{A}) = 1$. This means there is $[\und{i}] = [i_1,i_2,\dots,i_m] \in \supp{A}$ with $i_1 = 0$, $i_2 = -1$. Recall we have 
\begin{equation}
\ind \Tii{0}{t-1}\boxtimes \Rcalj{t}{A} \twoheadrightarrow A.
\end{equation}
By the Shuffle Lemma, the only way to have $[\und{i}] \in \supp{A}$ is if $\epch{-1}(\Rcalj{t}{A}) \geq 1$.
{\emph{Case 2}}: $\ell = 2$. Then $\epch{-1}(\Rcalj{1}{A}) = 2$, as we assumed $r(A) > 1$. Then there is $[\und{i}] \in \supp{A}$ with $i_1 = 0,$ $i_2 = -1$, $i_3 = -1$. Again, by the Shuffle Lemma, this is only possible if $\epch{-1}(\Rcalj{t}{A}) \geq 1$. 

Furthermore, when $t \equiv -1 \bmod \ell$ for $0 < t < r(A)$ we have $\epch{-1}(\Rcalj{t}{A}) = 2$ by the minimality of $r(A)$.
\end{proof}

\begin{lemma} \label{lemma-Rjump}
Let $A$ be a simple $R^{\Lambda_0}(\nu)$-module with $k = r(A)$, $\UnitModule \neq \Rcal{A} \in \rep{\Lambda_{-1}}$. Fix $j \in I$. Let $J = \jump{j}(\Rcal{A})$. If $k \neq j+1$ then for
$t \in \N$,
$0 \leq t \leq k$,
\begin{equation}
\jump{j}(\Rcalj{t}{A}) = \begin{cases}
J & t \not\equiv j+1 \bmod \ell \\
J+1 & t \equiv j+1 \bmod \ell.
\end{cases}
\end{equation}
If $k = j+1$ and $J \neq 0$, then for $0 \leq t \leq k$,
\begin{equation}
\jump{j}(\Rcalj{t}{A}) = \begin{cases}
J-1 & t \not\equiv j+1 \bmod \ell \\
J & t \equiv j+1 \bmod \ell.
\end{cases}
\end{equation}
If $k = j+1$ and $J = 0$, then for $0 < t < k$
\begin{equation}
\jump{j}(\Rcalj{t}{A}) = \begin{cases}
0 & t \not\equiv j+1 \bmod \ell \\
1 & t \equiv j+1 \bmod \ell.
\end{cases}
\end{equation}
and $\jump{j}(A) = 0$.
\end{lemma}

\begin{proof}
We will first prove the lemma in the case $A = \Tii{0}{m}$, where $m = |\nu|$. Then $k \leq \ell-1$. For $0 \leq t \leq k$ we have $\Rcalj{t}{A} = \T{t, t+1, \dots, m-1}$. From \eqref{jump-formula-triv}, $\jump{j} (\Rcalj{t}{A}) = \delta_{j,t-1} + \delta_{j,m}$ (recalling none of these modules are $\UnitModule$ by hypothesis). One can easily check the Lemma holds.

From now on, we assume $A$ is not a trivial module.

We now continue with the third case.
Suppose $k = j+1$ and $J = 0$. Then $\Rcal{A} = \Rcalj{k}{A} \in
\rep{\Lambda_{-1}}$, so $\epch{j}(\Rcal{A}) = \delta_{j,-1}$.
$\Rcalj{k-1}{A} = \ftilch{j}\Rcal{A} = \ftil{j}\Rcal{A}$
so $\epch{j}(\Rcalj{k-1}{A}) = \epch{j}(\Rcal{A}) + 1$
and $\ep{j}(\Rcalj{k-1}{A}) = \ep{j}(\Rcal{A}) + 1$.
In particular $\Rcalj{t}{A} \in
\rep{\Lambda_{t} + \Lambda_{-1}}$ but we may assume $\Rcalj{t}{A}
\notin \rep{\Lambda_{t}}$ or else by \cite{V07} this would force
$\Rcal{A}$ and $A$ itself to be trivial. Further
$\wti{j}(\Rcalj{k-1}{A}) = \wti{j}(\ftilch{j}\Rcal{A}) =
\wti{j}(\Rcal{A}) - 2$. Hence
\begin{align}
\jump{j} (\Rcalj{k-1}{A}) & = \epch{j}(\Rcalj{k-1}{A}) + \ep{j}(\Rcalj{k-1}{A}) + \wti{j}(\Rcalj{k-1}{A}) \\
\notag
& = \epch{j}(\Rcal{A}) + 1 + \ep{j}(\Rcal{A})+1 + \wti{j}(\Rcal{A})-2\\
\notag
& = 0
\end{align}
For $\Rcalj{k-2}{A},$
 \begin{equation}
\epch{j}(\Rcalj{k-2}{A}) = \delta_{j,k-2} + \delta_{j,-1} = \delta_{j,-1} = \epch{j}(\Rcal{A}). 
\end{equation}
Also 
\begin{equation}
\ep{j}(\Rcalj{k-2}{A}) = \ep{j}(\ftilch{j-1} \Rcalj{k-1}{A}) = \ep{j}(\Rcalj{k-1}{A})
\end{equation}
by Remark \ref{ep-and-ftil-for-i-neq-j}. 

If $k -2 \neq j+1$ (i.e. $\ell \neq 2$) then
\begin{align}
\jump{j}(\Rcalj{k-2}{A}) &= \epch{j}(\Rcal{A}) + \ep{j}(\Rcal{A}) + 1 + \wti{j}(\Rcal{A}) - 2 + 1 \\ & = 0.
\notag
\end{align}
Since $\ep{j}(\ftilch{i}B) = \ep{j}(B)$, $\epch{j}(\ftilch{i}B) = \epch{j}(B)$, and $\wti{j}(\ftilch{i}B) = \wti{j}(B)$ when $i \notin \{j-1,j,j+1\}$, similar computations show $\jump{j}(\Rcalj{t}{A}) = 0$ for $k-\ell < t \leq k$.

If $k - \ell > 0$, we check
\begin{align}
\jump{j}(\Rcalj{k-\ell}{A}) &= \delta_{j,j+1} + \delta_{j,-1} + \ep{j}(\ftilch{j+1}\Rcalj{k-\ell+1}{A}) + \wti{j}(\ftilch{j+1} \Rcalj{k-\ell+1}{A})\\
\notag
& = \jump{j}(\Rcalj{k-\ell+1}{A}) + 1 = 1.
\end{align}
(Note that when $\ell = 2$, $\epch{j}(\Rcalj{k-\ell}{A}) = \epch{j}(\Rcalj{k-\ell+1}{A}) -1$ but $\wti{j}(\ftilch{j+1}\Rcalj{k-\ell+1}{A}) = \wti{j}(\Rcalj{k-\ell+1}{A}) + 2$ so the equality still holds.)

Next $\jump{j}(\Rcalj{k-\ell-1}{A}) = \jump{j}(\ftilch{j} \Rcalj{k-\ell}{A}) = 1 - 1 = 0.$ Now all other inductive computations for $\jump{j}(\Rcalj{t}{A})$ are identical to the above computations, down to $t= 0$, for which $\Rcalj{0}{A} = A$. Here if $0 = t \equiv j+1  \bmod \ell$ then because $\epch{j}(A) = 0\neq \delta_{j,-1} + \delta_{j,0}$ we instead get $\jump{j}(A) = 0$.

The second case, $k = j+1$ but $J > 0$, is also very similar to the above. The only difference is that $\jump{j}(\Rcal{A}) \neq 0$ hence $\ep{j}(\Rcalj{k-1}{A}) = \ep{j}(\ftilch{j}\Rcal{A}) = \ep{j}(\Rcal{A})$. Regardless $\jump{j}(\Rcalj{k-1}{A}) = \jump{j}(\ftilch{j}\Rcal{A}) = J - 1$. We check 
\begin{align*}
\jump{j}(\Rcalj{k-2}{A}) & = \delta_{j,k-2} + \delta_{j,-1} +
\ep{j}(\ftilch{j-1}\Rcalj{k-1}{A}) + \wti{j}(\ftilch{j-1}\ftilch{j}
\Rcal{A}) \\
& = 0 + \epch{j}(\Rcal{A}) + \ep{j}(\Rcal{A}) +
\wti{j}(\Rcal{A}) - 2 - \langle h_j,\alpha_{j-1}\rangle \\
& = \begin{cases}
J -1 & \text{if} \;\; \ell \neq 2 \\
J & \text{if} \;\; \ell = 2.
\end{cases}
\end{align*}

Note in the case $\ell = 2$ that $k-2 \equiv j+1$, so this is consistent with the statement of the lemma.

The rest of the proof is identical to that in Case 1.

Finally we consider $k \not\equiv j+1 \bmod \ell$.
Letting $j_0 \in \MB{Z}$, $j_0 < k$ be maximal such that $j_0 \equiv j
\bmod \ell$, it is clear that $\jump{j}(\Rcalj{t}{A}) =
\jump{j}(\Rcal{A})$ for all $t > j_0 + 1$. 

Then 
\begin{align}
\jump{j} (\Rcalj{j_0 + 1}{A}) & = \delta_{j,j_0 +1} + \delta_{j,-1} + \ep{j}(\ftilch{j+1}\Rcalj{j_0+2}{A}) + \wti{j}(\ftilch{j+1}\Rcalj{j_0+2}{A})\\
& = 0 + \delta_{j,-1} + \ep{j}(\Rcalj{j_0+2}{A}) + \wti{j}(\Rcalj{j_0+2}{A}) - \langle h_j, \alpha_{j+1}\rangle \\
& = \jump{j}(\Rcalj{j_0 + 2}{A}) + 1 = J+1.
\end{align}
In the case $\ell \neq 2$ this follows as $\delta_{j,-1} = \epch{j}(\Rcalj{j_0 + 2}{A})$ and $\langle h_j, \alpha_{j+1} \rangle = -1$. In the case $\ell =2$, we have $\epch{j}(\Rcalj{j_0+2}{A}) = 1 + \delta_{j,-1}$ but $\langle h_j, \alpha_{i+1} \rangle = -2$. 

Next $\jump{j}(\Rcalj{j_0}{A}) = \jump{j}(\ftilch{j}\Rcalj{j_0+1}{A}) = \jump{j}(\Rcalj{j_0+1}{A})-1 = J$.

Also for $\ell \neq 2$, 
\begin{align}
\jump{j}(\Rcalj{j_0-1}{A}) &= \delta_{j,j-1} + \delta_{j,-1} + \ep{j}(\ftilch{j-1}\Rcalj{j_0}{A}) + \wti{j}(\ftilch{j-1}\Rcalj{j_0}{A})\\
& = \jump{j}(\Rcalj{j_0}{A}) = J
\end{align}
as $\epch{j}(\Rcalj{j_0}{A}) = 1 + \delta_{j,-1}$ and $\langle h_j,\alpha_{j-1} \rangle = -1$. We don't consider $\ell = 2$ as $j_0 -1 \equiv j_0 + 1$ and we have already considered that case. We note that the calculations of $\jump{j}(\Rcalj{t}{A})$ only depend on $t \bmod \ell$ and so we are done.
\end{proof}

\begin{proof} [Proof of Theorem \ref{thm-action-ftil}]
For ease of exposition we set $i = 0$ in the proof. In fact we will prove a slightly stronger statement, that when $\ep{j}(\Tii{0}{r(A)}) < \phcyc{-1}{j}(\Rcal{A})$ and $\pr{0} \ftil{j}A \neq \zero$ then $r(\ftil{j}A) = r(A)$ and for $0 < t \leq r(A)$, $\Rcalj{t}{\ftil{j}A} = \ftil{j}\Rcalj{t}{A}$. 

First note that in the case $\Rcal{A} = \UnitModule$ the theorem is obvious as $\ftil{j}\Rcal{A} = \Lii{j}$ and $\ftil{j} \Tii{0}{k} = \cosoc \ind \Tii{0}{k} \boxtimes \Lii{j}$. So from now on assume $\Rcal{A} \neq \UnitModule$.

{\emph{Case 1}}: Suppose $\ep{j}(\Tii{0}{k}) = 0$. In particular, $j \neq k-1$.

$\bullet$
{\emph{Case 1a}}: Suppose $\jump{j}(\Rcal{A}) = 0$. By Lemma \ref{lemma-Rjump}, 
\begin{equation}
\jump{j}(A) = \begin{cases}
0 & 0 \not\equiv j+1 \bmod \ell, \\
1 & 0 \equiv j+ 1 
\bmod \ell.
\end{cases}
\end{equation}
So when $j \neq -1$ as $\jump{j}(A) = 0$, $\pr{0}\ftil{j}A = \zero$ and so we need not consider this case. Hence we may assume $j = -1$. Further, as $\jump{-1}(\Rcal{A}) = 0$ and $\epch{-1}(\Rcal{A}) = 1$ we have
\begin{align}
\ind \Tii{0}{k} \boxtimes \Lii{-1} \boxtimes \Rcal{A} & \xrightarrow{\cong} \ind \Tii{0}{k} \boxtimes \Rcal{A} \boxtimes \Lii{-1} \\ 
& \twoheadrightarrow \ind A \boxtimes \Lii{-1} \twoheadrightarrow \ftil{-1}A
\end{align}
This implies $\pr{0} \ind \Tii{0}{k} \boxtimes \Lii{-1} \boxtimes \Lii{-1} \neq \zero$, forcing $k \equiv -1  \bmod \ell$ by Propositions \ref{interp-of-phcyc} and \ref{pr-facts}. But then we have
\begin{equation}
\ind \T{0, \dots, -2} \boxtimes \Lii{-1} \boxtimes \Rcal{A} \twoheadrightarrow \ftil{-1}A.
\end{equation}
Because 
\begin{equation}
\ind \T{0,\dots, -2} \boxtimes \Rcal{A} \twoheadrightarrow A
\end{equation}
we see $\ep{-1}(\Rcal{A}) = \ep{-1}(A) = \ep{-1}(\ftil{-1}A) - 1$. If we had 
\begin{equation}
\ind N \boxtimes \Rcal{A} \twoheadrightarrow \ftil{-1}A
\end{equation}
for any composition factor $N$ of $\ind \Tii{0}{k} \boxtimes \Lii{-1}$ other than $\ftil{-1}\Tii{0}{k}$, the Shuffle Lemma would yield $\ep{-1}(\ftil{-1}A) = \ep{-1}(\Rcal{A})$, a contradiction. Hence we have 
\begin{equation}
\ind \Tii{0}{k+1} \boxtimes \Rcal{A} = \ind \ftil{-1}\Tii{0}{k} \boxtimes \Rcal{A} \twoheadrightarrow \ftil{-1}A.
\end{equation}

$\bullet$
{\emph{Case 1b}}: Suppose that $\jump{j}(\Rcal{A}) = J > 0$. Again by Lemma \ref{lemma-Rjump} 
\begin{equation}
\jump{j}(A) = \begin{cases}
J & 0 \not\equiv j+1  \bmod \ell,\\
J+1 & 0 \equiv j+1 \not\equiv k  \bmod \ell.
\end{cases}
\end{equation}
Note $\ftil{j}\Rcal{A} \in \rep{\Lambda_{-1}}$ as $\jump{j}(\Rcal{A}) > 0$. 

We compute 
\begin{equation}
\ftil{j}\Rcalj{k-1}{A} = \ftil{j}\ftilch{k-1}\Rcal{A} = \ftilch{k-1}\ftil{j}\Rcal{A}
\end{equation}
as $j \neq k-1$ as in Case 1. Also, clearly $\ftilch{k-1}\ftil{j}\Rcal{A} \in \rep{\Lambda_{k-1} + \Lambda_{-1}}$ and $\ftilch{k-1}\ftil{j}\Rcal{A} \notin \rep{\Lambda_{-1}}$. Assume we have shown 
\begin{equation}
\ftil{j}\Rcalj{t}{A} = \ftilch{t} \dots \ftilch{k-1} \ftil{j} \Rcal{A} \in \rep{\Lambda_t + \Lambda_{-1}} \setminus \rep{\Lambda_{-1}}.
\end{equation}
Then we compute $\ftil{j}\Rcalj{t-1}{A} = \ftil{j}\ftilch{t-1}\Rcalj{t}{A}$. If $t - 1 \not\equiv j \bmod \ell$ this is equal to $\ftilch{t-1}\ftil{j}\Rcalj{t}{A} = \ftilch{t-1}\dots \ftilch{k-1}\ftil{j}\Rcal{A}$ by the inductive hypothesis. If instead $t \equiv j+1 \bmod \ell$ then 
\begin{equation}
\jump{j}(\Rcalj{t}{A}) = J+1> 1.
\end{equation}
So by Lemma \ref{commuting-functors}.\ref{commuting-f-same-i} 
\begin{align}
 \ftil{j}\Rcalj{t-1}{A} &= \ftil{j}\ftilch{j} \Rcalj{t}{A} = \ftilch{j} \ftil{j} \Rcalj{t}{A} = \ftilch{j}\ftilch{t} \dots \ftilch{k-1}\ftil{j}\Rcal{A} \\ &= \ftilch{t-1} \dots \ftilch{k-1} \ftil{j} \Rcal{A}.
\end{align}
By downwards induction $\ftil{j}A = \ftil{j}\Rcalj{0}{A} = \ftilch{0} \dots \ftilch{k-1} \ftil{j} \Rcal{A}$. Certainly $\ftil{j} \Rcalj{t-1}{A} \in \rep{\Lambda_{t-1} + \Lambda_{-1}} \setminus \rep{\Lambda_{-1}}$ as $\jump{j}(\Rcalj{t-1}{A}) > 0$ and $\Rcalj{t-1}{A} \in \rep{\Lambda_{t-1} + \Lambda_{-1}} \setminus \rep{\Lambda_{-1}}$ when $t-1 > 0$. And we already know $\ftil{j}A \in \rep{\Lambda_0}$. This completes the induction and furthermore shows
\begin{equation}
r(\ftil{j}A) = r(A), \quad\quad \Rcal{\ftil{j}A} = \ftil{j}\Rcal{A}
\end{equation}
as well as the stronger statement $\Rcalj{t}{\ftil{j}A} = \ftil{j}\Rcalj{t}{A}$ for $0 \leq t \leq k$. In other words we have
\begin{equation}
\ind \Tii{0}{k} \boxtimes \ftil{j}\Rcal{A} \rightarrow \ftil{j}{A}.
\end{equation}

{\emph{Case 2}}: Suppose that $\ep{j}(\Tii{0}{k}) = 1$. This is the
only other possibility as $\ep{j}(\Tii{0}{k}) \leq 1$ for all
$j \in I$. Note $k \equiv j+1 \bmod \ell$.

$\bullet$
{\emph{Case 2a}}: If $\phcyc{-1}{j}(\Rcal{A}) = 0$ then $\jump{j}(A) = 0$ by Lemma \ref{lemma-Rjump} so $\pr{0}\ftil{j}A = 0$ and we need not consider this case. 

$\bullet$
{\emph{Case 2b}}: If $\phcyc{-1}{j}(\Rcal{A}) = 1$ then again by Lemma \ref{lemma-Rjump} as $\Rcal{A} = \Rcalj{k}{A}$ and $k \equiv j+1 \bmod \ell$

\begin{equation}
\jump{j}(A) = \jump{j}(\Rcalj{0}{A}) = \begin{cases}
0 & 0 \not\equiv j+1 \bmod \ell \\
1 & 0 \equiv j+1 \bmod \ell.
\end{cases}
\end{equation}

So we need not consider this case unless $j = -1$. However we will
show, this case cannot arise
as we assumed $\Rcal{A} \neq
\UnitModule$. Note $\jump{j}(\Rcalj{k-1}{A}) =
\jump{-1}(\ftilch{-1}\Rcal{A}) = 0$ and $\Rcalj{k-1}{A} \in
\rep{2\Lambda_{-1}}$. Thus we have
\begin{align}
\ind \Tii{0}{k-1} \boxtimes \Lii{-1} \boxtimes \Rcalj{k-1}{A} & \xrightarrow{\cong} \ind \Tii{0}{k-1} \boxtimes \Rcalj{k-1}{A} \boxtimes \Lii{-1} \\ 
&\twoheadrightarrow \ind A \boxtimes \Lii{-1} \twoheadrightarrow \ftil{-1}A.
\end{align}
However $\pr{0} \ind \T{0, \dots, k-1} \boxtimes \Lii{-1} \boxtimes \Lii{-1} \boxtimes \Lii{-1} = \zero$ by Proposition \ref{triv-pr-Lemma}.\ref{k-eq-(-1)-part1}, which is a contradiction
to $\pr{0} \ftil{-1}A \neq \zero$.

$\bullet$
{\emph{Case 2c}}: When $\phcyc{-1}{j}(\Rcal{A}) = J > 1$ the argument is similar to Case $1b$.

\end{proof}

\subsection{Other types}
We can repeat the arguments above to prove similar
theorems in affine
type $B, C,$ and $D$, 
using the Kirillov-Reshetikhin crystal $B^{1,1}$ of appropriate type,
and  the modules corresponding to the
nodes studied in  \cite{V07} in place of $\Tii{i}{k}$.
This is work in progress \cite{KV}.


\bibliographystyle{amsplain}

\def\cprime{$'$} \def\cprime{$'$}
\providecommand{\bysame}{\leavevmode\hbox to3em{\hrulefill}\thinspace}
\providecommand{\MR}{\relax\ifhmode\unskip\space\fi MR }
\providecommand{\MRhref}[2]{%
  \href{http://www.ams.org/mathscinet-getitem?mr=#1}{#2}
}
\providecommand{\href}[2]{#2}

\end{document}